\documentclass[11pt]{amsart}
\usepackage{nccmath}

\usepackage{amsmath,amsthm,amssymb}
\usepackage{times}
\usepackage{enumerate}
\usepackage{color}
\usepackage{accents}
\newcommand{\ubar}[1]{\underaccent{\bar}{#1}}

\def\dfrac{\displaystyle\frac}
\def\dsum{\displaystyle\sum}

\newtheorem{prop}{Proposition}
\newtheorem{theo}[prop]{Theorem}
\newtheorem{lemm}[prop]{Lemma}
\newtheorem{coro}[prop]{Corollary}
\newtheorem{rmk}[prop]{Remark}

\newtheorem{claim}{Claim}

\newcommand{\be}{\begin{equation}}
\newcommand{\ee}{\end{equation}}
\newcommand{\lt}{\left}
\newcommand{\rt}{\right}
\newcommand{\al}{\alpha}
\newcommand{\e}{\epsilon}
\renewcommand{\leq}{\leqslant}
\renewcommand{\geq}{\geqslant}
\newcommand{\td}{\tilde}

\newcommand{\ka}{\kappa}
\newcommand{\la}{\kappa}
\newcommand{\s}{\sigma}
\newcommand{\R}{\mathbb{R}}
\newcommand{\M}{\mathcal{M}}
\newcommand{\bx}{\bar{x}}
\newcommand{\bn}{\bar{\nabla}}
\newcommand{\goto}{\rightarrow}
\newcommand{\dS}{\mathbb{S}}
\newcommand{\us}{u^*}
\newcommand{\uj}{u^J}
\newcommand{\ujs}{u^{J*}}
\newcommand{\hs}{h^*}

\newcommand{\vjs}{\varphi^{J*}}

\newcommand{\lus}{\ubar{u}^*}
\newcommand{\lujs}{\ubar{u}^{J*}}
\newcommand{\lu}{\ubar{u}}

\newcommand{\uu}{\bar{u}}
\newcommand{\ga}{\gamma}
\newcommand{\gas}{\gamma^*}
\newcommand{\vp}{\varphi}
\newcommand{\T}{\partial}
\newcommand{\p}{\partial}

\newcommand{\w}{w^*}
\newcommand{\ba}{\mathfrak{b}}
\newcommand{\z}{\mathbf{z}}
\newcommand{\B}{\mathcal{B}}
\newcommand{\F}{\mathcal{F}}

\newcommand{\conv}{\text{Conv}}
\numberwithin{equation}{section}

\begin{document}
\setlength{\baselineskip}{1.2\baselineskip}

\title[Constant Hessian curvature hypersurface in the Minkowski space]
{Entire spacelike hypersurfaces with constant $\sigma_k$ curvature in Minkowski space}

\author{Zhizhang Wang}
\address{School of Mathematical Science, Fudan University, Shanghai, China}
\email{zzwang@fudan.edu.cn}
\author{Ling Xiao}
\address{Department of Mathematics, University of Connecticut,
Storrs, Connecticut 06269}
\email{ling.2.xiao@uconn.edu}
\thanks{Research of the first author is supported by NSFC Grants No.11871161 and 11771103.}

\begin{abstract}
In this paper, we prove the existence of smooth, entire, strictly convex, spacelike, constant
$\s_k$ curvature hypersurfaces with prescribed lightlike directions in Minkowski space. This is equivalent to prove the existence of
smooth, entire, strictly convex, spacelike, constant $\s_k$ curvature hypersurfaces with prescribed Gauss map image.
We also show that there doesn't exist any entire, convex, strictly spacelike, constant $\s_k$ curvature hypersurfaces. Moreover, we generalize the result in \cite{RWX} and construct strictly convex, spacelike, constant $\s_k$ curvature hypersurface with bounded principal curvature, whose image of the Gauss map is the unit ball.

\end{abstract}

\maketitle

\section{Introduction}
Let $\R^{n, 1}$ be the Minkowski space with the Lorentzian metric
\[ds^2=\sum_{i=1}^{n}dx_{i}^2-dx_{n+1}^2.\]
In this paper, we study convex spacelike hypersurfaces with positive constant
$\sigma_k$ curvature in Minkowski space $\R^{n, 1}$. Here, $\s_k$ is the $k$-th elementary symmetric polynomial, i.e.,
 \[\s_k(\ka)=\sum\limits_{1\leq i_1<\cdots<i_k\leq n}\ka_{i_1}\cdots\ka_{i_k}.\]
Any such hypersurface can be written locally as a graph of a function
$x_{n+1}=u(x), x\in\R^n,$ satisfying the spacelike condition
\be\label{int1.1}
|Du|<1.
\ee

Treibergs started the research of constructing nontrivial entire spacelike CMC hypersurfaces in \cite{Tre}. He showed that
for any $f\in C^2(\dS^{n-1}),$ there is a spacelike, convex, CMC hypersurface $\M_u=\{(x, u(x))|x\in\R^n\}$ with bounded principal curvatures, such that as $|x|\goto\infty$,
 $u(x)\goto |x|+f\lt(\frac{x}{|x|}\rt)$. The result in \cite{Tre} was generalized by Choi-Treibergs in \cite{CT}, where they proved that
 for any closed set $\F\subset\dS^{n-1}$ and $f\in C^0(\F),$ there is a spacelike convex CMC hypersurface $\M_u,$ such that when $\frac{x}{|x|}\in \F,$
  $u(x)\goto |x|+f\lt(\frac{x}{|x|}\rt),$ as $|x|\goto\infty.$

One natural question to ask is: can we construct convex entire spacelike constant
$\s_k$ curvature hyersurfaces with prescribed lightlike directions $\F\subset\dS^{n-1}$ and an arbitrary $C^0$ perturbation on $\F$?

It turns out that this question is very difficult. There are only some partial results obtained so far. More specifically,
Li (see \cite{Li}) extended the result in \cite{Tre} to constant Gauss curvature. He proved that for any $f\in C^2(\dS^{n-1}),$ there is a spacelike constant Gauss curvature hypersurface $\M_u$ with bounded principal curvatures, such that as $|x|\goto\infty$,
 $u(x)\goto |x|+f\lt(\frac{x}{|x|}\rt)$. In 2006, Guan-Jian-Schoen \cite{GJS} showed that when
 $\F=\dS^{n-1}_+=\{x\in \dS^{n-1}| x_1\geq 0\}$ and $f\in C^\infty (\dS^{n-1}_+)$ satisfies some additional conditions, then
 there is a spacelike constant Gauss curvature hypersurface $\M_u$ such that when $\frac{x}{|x|}\in\dS^{n-1}_+,$
 $u(x)\goto |x|+f\lt(\frac{x}{|x|}\rt)$ as $|x|\goto\infty$. Later, Bayard-Schn\"urer (see \cite{BS}) showed that for any
closed subset $\F\subset\dS^{n-1}$ with $\T\F\in C^{1,1},$ there is a spacelike constant Gauss curvature hypersurface $\M_u$
such that when $\frac{x}{|x|}\in\F,$ $u(x)\goto |x|$ as $|x|\goto\infty$. Under a weaker assumption  on the regularity of $\F,$
Bayard (see \cite{Bay09}) also proved the existence of entire spacelike hypersurface
$\M_u$ with constant scalar curvature such that when $\frac{x}{|x|}\in\F,$ $u(x)\goto |x|$ as $|x|\goto\infty$. However, the hypersurface constructed in \cite{Bay09} may not be convex. Very recently, under the same settings as in \cite{Tre} and \cite{Li}, Ren-Wang-Xiao (see \cite{RWX}) solved the existence problem for constant $\s_{n-1}$ curvature hypersurfaces. In particular, for any $f\in C^2(\dS^{n-1}),$ they constructed a spacelike, strictly convex, constant $\s_{n-1}$ curvature hypersurface $\M_u$ with bounded principal curvatures, which satisfies as $|x|\goto\infty$,
 $u(x)\goto |x|+f\lt(\frac{x}{|x|}\rt).$
\subsection{Main result}
In this paper, we will investigate convex, entire, spacelike hypersurfaces of constant $\s_k$ curvature with prescribed lightlike directions. This is equivalent
to study convex, entire, spacelike hypersurfaces of constant $\s_k$ curvature with prescribed Gauss map image.
Our main Theorems are stated as follows.
\begin{theo}
\label{intth1.1}
Suppose $\mathcal{F}\subset\dS^{n-1}$ is the closure of an open subset and $\T\F\in C^{1,1}$. Then for $1<k<n$, there exists a smooth,
entire, spacelike, strictly convex hypersurface $\M_u=\{(x,u(x))| x\in\R^n\}$ satisfying
\be\label{int1.0}
\sigma_k(\ka[\M_u])=\binom{n}{k},
\ee
 where $\ka[\M_u]=(\kappa_1,\kappa_2,\cdots,\kappa_n)$ is the principal curvatures of $\M_u$.
 Moreover, when $\frac{x}{|x|}\in\F,$
 \be\label{int1.0'}
 u(x)\goto |x|,\,\mbox{as $|x|\goto\infty.$}
 \ee
Further, the Gauss map image of $\M_u$ is the convex hull $\conv(\F)$ of $\F$ in the unit disc.
\end{theo}

 In the process of proving Theorem \ref{intth1.1}, we obtain a Pogorelov type $C^2$ local estimate. A direct consequence of this estimate is the following
 nonexistence result.
\begin{coro}
\label{intcor1}
Suppose $\M_u=\{(x, u(x))| x\in\R^n\}$ is an entire, convex, spacelike hypersurface with constant $\s_k$ curvature, namely, it satisfies equation
\eqref{int1.0}. Moreover, we assume $\M_u$ is strictly spacelike, that is, there is some constant $\beta<1$ such that
$$|Du|\leq \beta<1,\,\,x\in\R^n.$$ Then, such $\M_u$ does not exist.
\end{coro}

We also generalize the existence Theorem in \cite{RWX} and prove
\begin{theo}
\label{intth1.2}
Given any $f\in C^2(\dS^{n-1})$, there is a unique, spacelike, strictly convex hypersurface $\M_u=\{(x, u(x))| x\in\R^n\}$
with bounded principle curvatures satisfying equation \eqref{int1.0}. Moreover,
\[u(x)\goto |x|+f\lt(\frac{x}{|x|}\rt), \,\,\mbox{as $|x|\goto\infty$}.\]
Furthermore, the Gauss map image of $\M_u$ is the open unit disc.
 \end{theo}
\subsection{Idea of the proof}
The natural idea of constructing entire spacelike hypersurfaces $\M_u$ that satisfy equations \eqref{int1.0} and \eqref{int1.0'}
is very straightforward.
First, we can use the entire constant Gauss curvature hypersurface constructed in \cite{BS} as our lower barrier $\lu$ and use the entire CMC
hypersurface constructed in \cite{CT} as the upper barrier $\uu$. Then, we look at the following Dirichlet problem
\be\label{dirichlet-minkowski}
\left\{
\begin{aligned}
\s_k(\ka[\M_u])&=\binom{n}{k}\,\,\text{in $B_R$}\\
u&=\vp_R\,\,\text{on $\partial B_R,$}
\end{aligned}
\right.
\ee
where $B_R\subset\R^n$ is a ball with radius $R$ and $\vp_R$ is some smooth function satisfies
$\lu|_{\p B_R}\leq\vp_R\leq\uu|_{\p B_R}.$ Finally, we prove the local $C^0,$ $C^1,$ and $C^2$ estimates for
the solution $u_R$ of the equation \eqref{dirichlet-minkowski}. These local estimates enable us to conclude that there exists a sequence
of solutions of \eqref{dirichlet-minkowski}, denoted by $\{u_{R_i}\}_{i=1}^{\infty},$ $R_i\goto\infty$ as $i\goto\infty,$ converging to an entire graph $u,$
and $u$ satisfies \eqref{int1.0}, \eqref{int1.0'}.

Unfortunately, the Dirichlet problem \eqref{dirichlet-minkowski} is unsolvable in Minkowski space for general $k.$ We have to find other approaches.
We will consider the following Dirichlet problem instead.
\be\label{main equation-ball}
\left\{
\begin{aligned}
F(\w\gas_{ik}\us_{kl}\gas_{lj})&=\frac{1}{\binom{n}{k}^{\frac{1}{k}}}\,\,\text{in $\td{F}$}\\
\us&=0\,\,\text{on $\F,$}
\end{aligned}
\right.
\ee
where $\w=\sqrt{1-|\xi|^2},$ $\gas_{ij}=\delta_{ij}-\frac{\xi_i\xi_j}{1+\w},$ $\us_{kl}=\frac{\p^2u}{\p\xi_k\p\xi_l},$ $\F\subset\dS^{n-1}$ as described in Theorem \ref{intth1.1},
$\td{F}$ is the convex hull of $\F$ in $B_1:=\{\xi\mid |\xi|<1\},$ and $F(\w\gas_{ik}\us_{kl}\gas_{lj})=\lt(\frac{\s_n}{\s_{n-k}}(\ka^*[\w\gas_{ik}\us_{kl}\gas_{lj}])\rt)^{1/k}.$
Here, $\ka^*[\w\gas_{ik}\us_{kl}\gas_{lj}]=(\ka^*_1, \cdots, \ka^*_n)$ are the eigenvalues of the matrix $(\w\gas_{ik}\us_{kl}\gas_{lj}).$
The advantage of studying \eqref{main equation-ball} is that it restricts us to convex solutions.
In the Subsection \ref{lt} and Section \ref{gm}, we will illustrate that if $\us$ is a solution of \eqref{main equation-ball}, then the Legendre transform of
$\us,$ denoted by $u$, satisfies \eqref{int1.0} and \eqref{int1.0'}. However, equation \eqref{main equation-ball} is a degenerate equation
which cannot be solved directly.

We need to study the following approximating problems
\be\label{dirichlet-ball}
\left\{
\begin{aligned}
F(\w\gas_{ik}\ujs_{kl}\gas_{lj})&=\frac{1}{\binom{n}{k}^{\frac{1}{k}}}\,\,\text{in $\td{F}_J$}\\
\us&=\vjs\,\,\text{on $\p\td{F}_J,$}
\end{aligned}
\right.
\ee
where $\{\td{F}_J\}_{J=1}^\infty$ is a sequence of smooth convex set in $\td{F}$ that approaches $\td{F},$
$\vjs=\lus|_{\p\td{F}_J},$ and $\lus$ is the Legendre transform of $\lu.$ Despite the equation \eqref{dirichlet-ball}
is no longer degenerate, there is no known existence result for it either. The main difficulty is to obtain the global $C^2$ estimate.
As we already know, in order to obtain the global $C^2$ estimate we need to get a $C^2$ boundary estimate first.
However, the $C^2$ boundary estimate in this case is very challenging.
Recall that our function $F=\lt(\frac{\s_n}{\s_{n-k}}\rt)^{1/k}.$ Therefore, to obtain the $C^2$ boundary estimate we have to get estimates
on both $u_{\al n}$ and $u_{nn},$ where $u_{\al n}$ is the tangential normal mixed derivative at the boundary and $u_{nn}$ is the double normal derivative.
In \cite{Tru}, Trudinger was able to obtain the $C^2$ boundary estimate for the Hessian equations of the form $\frac{\s_n}{\s_{n-k}}(\ka[D^2u])=\psi.$
Here, while we are able to estimate $u_{n\al},$ due to the complication of $\w\gas_{ik}\ujs_{kl}\gas_{lj}$,
we fail to adapt his method to obtain the estimate on $u_{nn}.$
It's desirable to  find a simpler equivalent expression for equation \eqref{main equation-ball}

It's well known that the Gauss map $G:\M\goto\mathbb{H}^n(-1)$ maps a strictly convex spacelike hypersurface $\M$ to the hyperbolic space $\mathbb{H}^n(-1).$
We will see in subsection \ref{gm} that the solvability of \eqref{dirichlet-ball} is equivalent to the solvability of the following equation:
\be\label{dirichlet-hyperbolic}
\left\{
\begin{aligned}
F(v_{ij}-v\delta_{ij})&=\frac{1}{\binom{n}{k}^{\frac{1}{k}}},\,\,\mbox{in $U_J$}\\
v&=\frac{\vjs(\xi)}{\sqrt{1-|\xi|^2}},\,\,\mbox{on $\partial U_J.$}
\end{aligned}
\right.
\ee
where $v_{ij}=\bn_i\bn_jv$ denotes the covariant derivative with respect to the hyperbolic metric,  $U_J=P^{-1}(\td{F}_J)\subset\mathbb{H}^n(-1),$ and $P: \mathbb{H}^n\goto B_1$ is the projection of $\mathbb{H}^n.$ Moreover, we have that the eigenvalues of the matrix
$(\w\gas_{ik}\ujs_{kl}\gas_{lj})$ are the same as the eigenvalues of the matrix $(\bn_i\bn_j v-v\delta_{ij}).$ Therefore, we will study the $C^2$ estimates
for equation \eqref{dirichlet-hyperbolic}. Surprisingly, as nice as \eqref{dirichlet-hyperbolic} may seem to be, the $C^2$ bound for $v$ is very tricky to obtain. In fact, fully nonlinear equations of the form of equation \eqref{dirichlet-hyperbolic} in Riemannian manifold have been studied in \cite{Guan}. However, our functional $F$ doesn't meet all conditions that are required in \cite{Guan}. Therefore, we need to develop new ways to obtain the $C^2$ global estimate. The difficult part in this model is to construct an auxiliary function that will be needed to estimate the tangential normal mixed derivatives. We overcome this difficulty by a key observation that essentially connects $\bn_i\bn_j$ with $\p_{\xi_i}\p_{\xi_j}$ (see Lemma \ref{c2blem1.1}), so that we can utilize the convexity of $\td{F}_J$ to construct an auxiliary function we need in $\mathbb{H}^n$.

The last major obstacle is the $C^1$ local estimate. In \cite{Bay} and \cite{BS}, Bayard and Bayard-Schn\"urer first observed that, if there exists a spacelike function $\psi$ satisfying $\psi<\lu$ in a compact set $K\subset\R^n$ and $\psi>\uu$ as $|x|\goto\infty.$ Then, there is a local $C^1$ estimate for the solution
$u_R$ of equation \eqref{dirichlet-minkowski}, where $R>0$ large such that $B_R\supset K.$ It's clear that $\psi$ serves as a cutoff function here.
In the constant Gauss curvature case (see \cite{BS}), the lower barrier $\lu$ and upper barrier $\uu$ constructed by Bayard-Schn\"urer satisfy $\lu-\uu\goto 0$ as $|x|\goto \infty.$ Therefore, a rescaling of $\lu$ yields a perfect cutoff function. However, in the constant $\s_k$ curvature case, one can not find barrier functions $\lu$ and $\uu$ satisfying $\lu-\uu\goto 0$ as $|x|\goto \infty.$ Moreover, when $\frac{x}{|x|}\in\F,$ as $|x|\goto \infty,$ $|D\lu|$ and $|D\uu|\goto 1;$ while we need to make sure that $\psi$ is spacelike. Thus, to construct a cutoff function $\psi,$ we need to carefully analyze the asymptotic behavior of our $\lu$ and $\uu.$  The construction is very delicate (see Lemma \ref{lc1lem3}).

\begin{rmk}
After this paper was done, we discovered that in \cite{Bay09} Bayard successfully constructed a spacelike cutoff function $\psi$ by rescaling CMC hypersurfaces.
His construction also overcomes the problem that for some directions $\theta\in\dS^{n-1},$ the barrier function $\lim\limits_{r\goto\infty}(\lu(r\theta)-\uu(r\theta))\nrightarrow0.$ However, the advantage of our construction is that our $\psi$ has a very explicit formula (see \eqref{cutoff-function}), which enables us to construct prescribed curvature hypersurfaces with nonzero data in the lightlike directions, i.e., for $f\in C^2(\dS^{n-1})$ and $\F\subset\dS^{n-1},$ when $\frac{x}{|x|}\in\F,$ $u(x)-|x|\goto f\lt(\frac{x}{|x|}\rt)$ as $|x|\goto\infty.$ We will include this result in an upcoming paper.
\end{rmk}
\subsection{Outline} The organization of the paper is as follows. In Section \ref{pre}, we introduce some basic formulas and notations. In particular,
we investigate strictly convex, spacelike hypersurfaces under the Gauss map and the Legendre transform respectively. We summarize properties of the Gauss map in Section \ref{gm}. In Section \ref{cb}, we construct sub- and super- solutions of equation \eqref{int1.0}. We also review properties of semitroughs. These properties give us a thorough understanding of the asymptotic behavior of the sub- and super- solutions which will be needed in Section \ref{cv}. The solvability of equation \eqref{dirichlet-ball} is discussed in Section \ref{cs}. In Section \ref{cv}, we prove the local $C^1$ and $C^2$ estimates, which leads to proofs of our main theorems.

\section{Preliminaries}
\label{pre}
In this section, we will derive some basic formulas for the geometric quantities of spacelike hypersurfaces in Minkowski space $\R^{n, 1}.$
We first recall that the Minkowski space $\R^{n,1}$ is $\R^{n+1}$ endowed with the Lorentzian metric
$$ds^2=dx_1^2+\cdots dx_{n}^2-dx_{n+1}^2.$$
Throughout this paper, $\lt<\cdot, \cdot\rt>$ denotes the inner product in $\R^{n,1}$.

\subsection {Vertical graphs in $\R^{n, 1}$}
\label{vg}
A spacelike hypersurface $\M$ in $\R^{n, 1}$ is a codimension one submanifold whose induced metric
is Riemannian. Locally $\M$ can be written as a graph
\[\M_u=\{X=(x, u(x))| x\in\R^n\}\]
satisfying the spacelike condition \eqref{int1.1}. Let
$\mathbf{E}=(0, \cdots, 0, 1),$ then the height function of $\M$ is $u(x)=-\lt<X, \mathbf{E}\rt>.$ It's easy to see that the induced metric and second fundamental form of $\M$ are given
by
$$g_{ij}=\delta_{ij}-D_{x_i}uD_{x_j}u, \ \  1\leq i,j\leq n,$$
and
\[h_{ij}=\frac{u_{x_ix_j}}{\sqrt{1-|Du|^2}},\]
while the timelike unit normal vector field to $\M$ is
\[\nu=\frac{(Du, 1)}{\sqrt{1-|Du|^2}},\]
where $Du=(u_{x_1}, \cdots, u_{x_n})$ and $D^2u=\lt(u_{x_ix_j}\rt)$ denote the ordinary gradient and Hessian of $u$,
respectively. By a straightforward calculation, we have the principle curvatures of $\M$ are eigenvalues of the symmetric matrix
$A=(a_{ij}):$
\[a_{ij}=\frac{1}{w}\ga^{ik}u_{kl}\ga^{lj},\]
where $\ga^{ik}=\delta_{ik}+\frac{u_iu_k}{w(1+w)}$ and $w=\sqrt{1-|Du|^2}.$ Note that $(\ga^{ij})$ is invertible with inverse
$\ga_{ij}=\delta_{ij}-\frac{u_iu_j}{1+w},$ which is the square root of $(g_{ij}).$

Let $\mathcal{S}$ be the vector of $n\times n$ symmetric matrices and
\[\mathcal{S}_+=\{A\in \mathcal{S}: \lambda(A)\in \Gamma_n\},\]
where $\Gamma_n:=\{\lambda\in\R^n:\,\,\mbox{each component $\lambda_i>0$}\}$ is the convex cone, and $\lambda(A)=(\lambda_1, \cdots, \lambda_n)$ denotes the eigenvalues of $A.$
Define a function $F$ by
\[F(A)=\sigma^{\frac{1}{k}}_k(\lambda(A)),\,\, A\in\mathcal{S}_+,\]
then \eqref{int1.0} can be written as
\be\label{main equation graph}
F\lt(\frac{1}{w}\ga^{ik}u_{kl}\ga^{lj}\rt)=\binom{n}{k}^{\frac{1}{k}}.
\ee
Note that, in fact the function $F$ is well defined on $\mathcal{S}_k=\{A\in \mathcal{S}: \lambda(A)\in \Gamma_k\},$ where
$\Gamma_k$ is the G{\aa}rding cone (see \cite{CNS}).
However, in this paper, we only study strictly convex hypersurfaces, thus we restrict ourselves to $\mathcal{S}_+.$
Throughout this paper we denote
\[F^{ij}(A)=\frac{\p F}{\p a_{ij}}(A),\,\,F^{ij, kl}=\frac{\p^2 F}{\p a_{ij}\p a_{kl}}.\]

One important example of the spacelike hypersurface with constant mean curvature is the hyperboloid
\[u(x)=\lt(\frac{n^2}{H^2}+\sum_{i=1}^{n}x_i^2\rt)^{1/2},\]
which is umbilic, i.e., it satisfies $\kappa_1=\kappa_2=\cdots=\kappa_n=\frac{H}{n}.$ Other examples of spacelike CMC hypersurfaces include hypersurfaces of revolution,
in which case the graph takes the form $u(x)=\sqrt{f(x_1)^2+|\bar{x}|^2},\,\,x=(x_1, \bar{x})=(x_1, \cdots, x_n)\in\R^n,$ where $f$ is a function only depending on $x_1$. In Section \ref{cb}, we will discuss properties of CMC hypersurfaces of this type in details.

Now, let $\{\tau_1,\tau_2,\cdots,\tau_n\}$ be a local orthonormal frame on $T\M$. We will use $\nabla$ to denote
the induced Levi-Civita connection on $\M.$ For a function $v$ on $\M$, we denote $v_i=\nabla_{\tau_i}v,$ $v_{ij}=\nabla_{\tau_i}\nabla_{\tau_j}v,$ etc.
In particular, we have
\[|\nabla u|=\sqrt{g^{ij}u_{x_i}u_{x_j}}=\frac{|Du|}{\sqrt{1-|Du|^2}}.\]

Using normal coordinates, we also need the following well known fundamental equations for a hypersurface $\M$ in $\R^{n, 1}:$
\begin{equation}\label{Gauss}
\begin{array}{rll}
X_{ij}=& h_{ij}\nu\quad {\rm (Gauss\ formula)}\\
(\nu)_i=&h_{ij}\tau_j\quad {\rm (Weigarten\ formula)}\\
h_{ijk}=& h_{ikj}\quad {\rm (Codazzi\ equation)}\\
R_{ijkl}=&-(h_{ik}h_{jl}-h_{il}h_{jk})\quad {\rm (Gauss\ equation)},\\
\end{array}
\end{equation}
where $R_{ijkl}$ is the $(4,0)$-Riemannian curvature tensor of $\M$, and the derivative here is covariant derivative with respect to the metric on $\M$.
It is clear that the Gauss formula and the Gauss equation in \eqref{Gauss} are different from those in Euclidean space.
Therefore, the Ricci identity becomes,
\begin{equation}\label{a1.2}
\begin{array}{rll}
h_{ijkl}=& h_{ijlk}+h_{mj}R_{imlk}+h_{im}R_{jmlk}\\
=& h_{klij}-(h_{mj}h_{il}-h_{ml}h_{ij})h_{mk}-(h_{mj}h_{kl}-h_{ml}h_{kj})h_{mi}.\\
\end{array}
\end{equation}
\subsection{The Gauss map}
\label{gg}
Let $\M$ be an entire, strictly convex, spacelike hypersurface, $\nu(X)$ be the timelike unit normal vector to $\M$
at $X.$ It's well known that the hyperbolic space $\mathbb{H}^{n}(-1)$ is canonically embedded in $\R^{n, 1}$
as the hypersurface
\[\lt<X, X\rt>=-1,\,\, x_{n+1}>0.\]
By parallel translating to the origin we can regard $\nu(X)$
as a point in $\mathbb{H}^n(-1).$ In this way, we define the Gauss map:
\[G: \M\rightarrow \mathbb{H}^n(-1);\,\, X\mapsto\nu(X).\]
If we take the hyperplane $\mathbb{P}:=\{X=(x_1, \cdots, x_{n}, x_{n+1}) |\, x_{n+1}=1\}$ and consider the projection of
$\mathbb{H}^n(-1)$ from the origin into $\mathbb{P}.$ Then $\mathbb{H}^n(-1)$ is mapped in
a one-to-one fashion onto an open unit ball $B_1:=\{\xi\in\R^n |\, \sum\xi^2_k<1\}.$ The map
$P$ is given by
\[P: \mathbb{H}^n(-1)\rightarrow B_1;\,\,(x_1, \cdots, x_{n+1})\mapsto (\xi_1, \cdots, \xi_n),\]
where $x_{n+1}=\sqrt{1+x_1^2+\cdots+x_n^2},$ $\xi_i=\frac{x_i}{x_{n+1}}.$
We will call the map $P\circ G: \M\rightarrow B_1$ the Gauss map and denote it
by $G$ for the sake of simplicity.

Next, let's consider the support function of $\M.$ We denote
\[v:=\lt<X, \nu\rt>=\frac{1}{\sqrt{1-|Du|^2}}\lt(\sum_ix_i\frac{\partial u}{\partial x_i}-u\rt).\]
Let $\{e_1, \cdots, e_n\}$ be an orthonormal frame on $\mathbb{H}^n.$ We will also denote
$\{e^*_1, \cdots, e^*_n\}$ the push-forward of $e_i$ by the Gauss map $G.$ Similar to the convex geometry case,
we denote
\[\Lambda_{ij}=v_{ij}-v\delta_{ij}\]
the hyperbolic Hessian. Here $v_{ij}$ denote the covariant derivatives with respect to the hyperbolic metric.

Let $\bar{\nabla}$ be the connection of the ambient space. Then, we have
 $$v_i=\bar{\nabla}_{e^*_i}X\cdot \nu+X\cdot \bar{\nabla}_{e_i}\nu=X\cdot e_i,$$ this implies
 $$X=\sum_iv_ie_i-v\nu.$$ Note that $\lt<\nu, \nu\rt>=-1,$
thus we have,
\begin{eqnarray}
\bar{\nabla}_{e_j^*}X&=&\sum_k(e_j(v_k)e_k+v_k\bar{\nabla}_{e_j}e_k)-v_j\nu-v\bar{\nabla}_{e_j}\nu \\
&=&\sum_k(e_j(v_k)e_k+v_k\nabla_{e_j}e_k+v_k\delta_{kj}\nu)-v_j\nu-ve_j\nonumber\\
&=&\sum_k\Lambda_{kj}e_k\nonumber,\\
g_{ij}&=&\bar{\nabla}_{e^*_i}X\cdot \bar{\nabla}_{e^*_j}X=\sum_k\Lambda_{ik}\Lambda_{kj},\\
h_{ij}&=&\bar{\nabla}_{e^*_i}X\cdot \bar{\nabla}_{e_j}\nu=\Lambda_{ij}.
 \end{eqnarray}
This implies that the eigenvalues of the hyperbolic Hessian are the curvature radius of $\M$. That is,
if the principal curvatures of $\M$ are $(\la_1, \cdots, \la_n),$ then the eigenvalues of the hyperbolic Hessian
are $\lt(\la_1^{-1}, \cdots, \la_n^{-1}\rt).$
Therefore, equation \eqref{int1.0} can be written as
\be\label{main equation hyperbolic}
F(v_{ij}-v\delta_{ij})=\frac{1}{\binom{n}{k}^{\frac{1}{k}}},
\ee
where $F(A)=\lt[\frac{\s_n}{\s_{n-k}}(\lambda(A))\rt]^{\frac{1}{k}}.$
Moreover, it is clear that
\be\label{gg1.1}
\left(\bar{\nabla}_{e_j}\bar{\nabla}_{e_i}\nu\right)^{\bot}=\delta_{ij}\nu,
\ee
this yields, for $k=1,2\cdots,n+1$,
\be\label{gg1.2}
 \nabla_{e_j}\nabla_{e_i}x_k=x_k\delta_{ij},
\ee
where $x_k$ is the coordinate function. These properties will be used in Subsection \ref{c2g}.

\subsection{Legendre transform}
\label{lt}
Suppose $\M$ is an entire, stictly convex, spacelike hypersurface.
Then $\M$ is the graph of a convex function
\[x_{n+1}=-\lt<X, \mathbf{E}\rt>=u(x_1, \cdots, x_n),\]
where $\mathbf{E}=(0, \cdots, 0, 1).$
Introduce the Legendre transform
\[\xi_i=\frac{\T u}{\T x_i},\,\, u^*=\sum x_i\xi_i-u.\]
From the theory of convex bodies we know that
\[\Omega=\lt\{(\xi_1, \cdots, \xi_n)| \xi_i=\frac{\partial u}{\partial x_i}(x), x\in\R^n\rt\}\]
is a convex domain.

In particular, let $u(x)=\sqrt{1+|x|^2},$ $x\in\R^n,$ be a hyperboloid with principal curvatures being equal to $1.$ Then it's Legendre transform is
$\us(\xi)=-\sqrt{1-|\xi|^2},$ $\xi\in B_1.$

Next, we calculate the first and the second fundamental forms in terms of $\xi_i$. Since
\[x_i=\frac{\T\us}{\T \xi_i},\,\, u=\sum\xi_i\frac{\T\us}{\T\xi_i}-\us,\]
and it is well known that
$$\left(\frac{\T^2 u}{\T x_i\T x_j}\right)=\left(\frac{\T^2 \us}{\T \xi_i\T \xi_j}\right)^{-1}.$$
We have, using the coordinate $\{\xi_1,\xi_2,\cdots,\xi_n\}$, the first and the second fundamental forms can be rewritten as:
$$g_{ij}=\delta_{ij}-\xi_i\xi_j, \text{ and\,\,  } h_{ij}=\frac{u^{* ij}}{\sqrt{1-|\xi|^2}},$$
where $\lt(u^{* ij}\rt)$ denotes the inverse matrix of $(\us_{ij})$ and $|\xi|^2=\sum_i\xi_i^2$. Now, let $W$ denote the Weingarten matrix of $\M,$ then
$$(W^{-1})_{ij}=\sqrt{1-|\xi|^2}g_{ik}\us_{kj}.$$

From the discussion above, we can see that if $\M_u=\{(x, u(x)) | x\in\R^n\}$ is an entire, strictly convex, spacelike
hypersurface satisfying $\sigma_{k}(\la[\M])=\binom{n}{k},$ then the Legendre transform of
$u$ denoted by $\us,$ satisfies
\be\label{main equation legendre}
F(\w\gas_{ik}\us_{kl}\gas_{lj})=\lt[\frac{\sigma_n}{\sigma_{n-k}}(\la^*[\w\gas_{ik}\us_{kl}\gas_{lj}])\rt]^{\frac{1}{k}}=\frac{1}{\binom{n}{k}^{\frac{1}{k}}}.
\ee
Here, $\w=\sqrt{1-|\xi|^2}$ and $\gas_{ij}=\delta_{ij}-\frac{\xi_i\xi_j}{1+\w}$ is the square root of the matrix $g_{ij}.$

\bigskip

\section{The Gauss map image of an entire spacelike hypersurface of constant $\sigma_k$ curvature}
\label{gm}
In order to explain our results clearly,  we recall some results from \cite{CT} concerning the Gauss map of an entire spacelike constant mean curvature hypersurface. With no modification, we can show that, these results also hold for strictly convex constant $\sigma_k$ curvature hypersurfaces. For readers' convenience, in the following, we will state these results for strictly convex constant $\sigma_k$ curvature hypersurfaces.

\begin{lemm}\label{gmlem1}(see Lemma 4.1 in \cite{CT})
Let $u$ be a strictly convex spacelike function on $\R^n$. Then the blowdown of $u$,
\be\label{gm1.1}
V_u(x)=\lim\limits_{r\goto\infty}\frac{u(rx)}{r}
\ee
exists for all $x,$ and $V_u$ is an achronal, positive, homogeneous degree one, and null function.
\end{lemm}

 Following \cite{CT}, we will denote the class of all null achronal positive homogenous degree one convex functions on
 $\R^n$ by $\mathcal{Q}.$

\begin{lemm}\label{gmlem2}(See Lemma 4.3 in \cite{CT})
Let $E$ be a closed subset of $\dS^{n-1}.$ Then the function on $\R^n$ given by
\[V_E(x)=\sup\limits_{\xi\in E}\xi\cdot x,\]
where the inner product is the usual one from $\R^n,$ is convex, homogeneous and null. In fact,
the mapping $\mathfrak{F}\goto\mathcal{Q}$ given by $E\goto V_E$ is a one-to-one correspondence.
\end{lemm}

More specifically, let $\M_u=\{(x, u(x))|x\in\R^n\}$ be a strictly convex, entire, spacelike hypersurface satisfying $\sigma_k(\ka[\M_u])=\binom{n}{k}.$ Then the blowdown of $u(x)$
is determined by its lightlike directions
\[L_u=\{\xi\in\dS^{n-1}: V_u(x)=1\}.\]
Moreover, we have
\begin{lemm}\label{gmlem3}(See Lemma 4.5 and Lemma 4.6 in \cite{CT})
Let $\M_u=\{(x, u(x))|x\in\R^n\}$ be a strictly convex, entire, spacelike hypersurface satisfying $\sigma_k(\ka[\M_u])=\binom{n}{k}.$
Then $$Du(\R^n)=\conv(L_u),$$
where $\conv(L_u)$ is the convex hull of $L_u$ in $B_1.$
\end{lemm}

\bigskip
Thus, using the Splitting Theorem referred in Remark 2, one obtains a description of the Gauss map image for entire, spacelike, convex,
constant $\s_k$ curvature hypersurfaces:
\begin{theo}
(See Theorem 4.8 in \cite{CT})
Let $\M_u=\{(x, u(x))|x\in\R^n\}$ be a convex, entire, spacelike hypersurface satisfying $\sigma_k(\ka[\M_u])=\binom{n}{k}.$ If $l,$  $k\leq l\leq n,$ is the largest integer for which $\conv(L_u)\cap A_l$ has nonempty interior in $A_l$, for some $A_l$, which is a $l$-plane passing through the origin in $\mathbb{R}^n.$ Then $\M_u$ splits, up to ambient isometry, as $\M_u=\M_u^l\times \mathbb{R}^{n-l}$ intrinsically, where $\M_u^l$ is a strictly convex hypersurface in $\mathbb{R}^{l,1}$. In particular, if $L_u$ is with full rank, i.e. contained in no $A_l$, $l<n$, then $u$ is strictly convex.
\end{theo}

\bigskip
\section{The construction of barriers}
\label{cb}
In this section, we will describe known examples of entire spacelike constant Gauss curvature hypersurfaces (see \cite{BS})
and constant mean curvature hypersurfaces (see \cite{CT}). We will use these hypersurfaces as our barriers. We will also recall the properties of
semitroughs of constant Gauss curvature and constant mean curvature. A thorough understanding of semitroughs can help us to understand the behavior of the barrier functions at infinity ($|x|\goto\infty$). This will be needed in proving the local $C^1$ estimates (see Section \ref{cv}).

\subsection{Semitroughs}
\label{semitroughs}
Let's first recall the properties of the \textit{standard semitrough} for the constant Gauss curvature (see \cite{GJS}) and
constant mean curvature hypersurfaces (see \cite{CT}): it's a function $\z$
of the form
\[\z(x)=\sqrt{f^2(x_1)+|\bx|^2},\,\,\bx=(x_2, \cdots, x_n),\]
whose graph $\M_{\z}$ has constant $\s_n$ and $\s_1$ curvature respectively. Moreover,
\[D\z(\mathbb{R}^n)=\lt(\frac{ff'}{\z}, \frac{\bx}{\z}\rt)=\{\xi\in B_1:\xi_1>0\}:={B^+}.\]
We will use $\z^1$ to denote the standard semitrough that satisfies
$\s_1(\ka[\M_{\z^1}])=n$; and use $\z^n$ to denote the standard semitrough that satisfies
$\s_n(\ka[\M_{\z^n}])=1$.
From Lemma 5.1 of \cite{CT} and Lemma 2.2 of \cite{GJS} we know that for $\theta=(\theta_1,\theta_2,\cdots, \theta_n)\in\mathbb{S}^{n-1}$, $\lambda=1,n,$
\be\label{cb1.1}
\lim\limits_{r\goto\infty}{\left(\z^\lambda(r\theta)-V_{\bar{B}^+}(r\theta)\right)}=
\left\{
\begin{aligned}
&l_\lambda,\,\,\theta\bot\T_0\bar{B}^+:=\{\xi\in\bar{B}_1, \xi_1=0\}\text{ and } \theta_1=-1\\
&0,\,\,\text{elsewhere.}
\end{aligned}
\right.
\ee
Here and in the rest of this paper, we denote $l_1=\frac{n-1}{n}$ and $l_n=0.$ Following notations of \cite{CT} and \cite{BS},
we denote {$V_{\bar{E}}(x)=\sup\limits_{\xi\in\bar{E}}x\cdot\xi,$ for any $E\subset B_1$} and denote by $d_S$ the natural distance on $\dS^{n-1}.$
For any $x, y\in\dS^{n-1}$ we have \[d_S(x, y)=\arccos(x\cdot y)\in[0, \pi],\]
where the dot stands for the canonical scalar product in $\R^n.$ A ball in $\dS^{n-1}$ is a ball in the metric
space $(\dS^{n-1}, d_S),$ i.e., a set
\[\B=\{x\in\dS^{n-1}: d_S(x, x_0)<\delta\},\]
where $x_0\in\dS^{n-1}$ and $\delta>0$ is the radius of $\B,$ also denoted by $\delta(\B).$

Applying Lorentz transformation to $\z$
\be\label{cb1.2}
\left\{
\begin{aligned}
x_1'&=\frac{x_1-\alpha x_{n+1}}{\sqrt{1-\alpha^2}}\\
x_i'&=x_i\\
x'_{n+1}&=\frac{x_{n+1}-\alpha x_1}{\sqrt{1-\alpha^2}},
\end{aligned}
\right.
\ee
we get \[\td{\z}(x_1',\cdots, x_n')=x_{n+1}'=\frac{\z-\alpha x_1}{\sqrt{1-\alpha^2}},\]
here $\alpha\in(-1, 1).$
By a straightforward calculation we obtain, for $i\geq 2,$
\[0=\frac{\T x_1}{\T x_i'}-\alpha\frac{\T \z}{\T x_1}\frac{\T x_1}{\T x_i'}-\alpha\z_i,\]
this gives us
\[\frac{\alpha\z_i}{1-\alpha\z_1}=\frac{\T x_1}{\T x_i'}.\]
When $i=1$
\[\sqrt{1-\alpha^2}=(1-\alpha\z_1)\frac{\T x_1}{\T x_1'},\]
which implies
\[\frac{\T x_1}{\T x_1'}=\frac{\sqrt{1-\alpha^2}}{1-\alpha\z_1}.\]

Therefore, we have
\[\frac{\p\td{\z}}{\p x_1'}=\frac{(\z_1-\alpha)}{\sqrt{1-\alpha^2}}\cdot\frac{\sqrt{1-\alpha^2}}{1-\alpha\z_1}=\frac{\z_1-\alpha}{1-\alpha\z_1},\]
and for $i\geq 2,$
\begin{align*}
\frac{\p\td{\z}}{\p x_i'}&=\frac{1}{\sqrt{1-\alpha^2}}\lt(\z_1\frac{\T x_1}{\T x_i'}+\z_i-\alpha\frac{\T x_1}{\T x_i'}\rt)\\
&=\frac{1}{\sqrt{1-\alpha^2}}\lt[\frac{\alpha(\z_1-\alpha)\z_i}{1-\alpha\z_1}+\z_i\rt]\\
&=\frac{\z_i\sqrt{1-\alpha^2}}{1-\alpha\z_1}.\\
\end{align*}
This yields
\[D\td{\z}(\mathbb{R}^n)=\lt(\frac{\z_1-\alpha}{1-\alpha\z_1}, \sqrt{1-\alpha^2}\frac{\z_i}{1-\alpha\z_1}\rt)(\mathbb{R}^n):=\{\xi\in B_1, \xi_1>-\alpha\}.\]
From the above calculation we can see that, for every closed ball $\bar{\B}$ of $\dS^{n-1},$ by a rotation of coordinates, there exists an entire spacelike function
$\z^\lambda_{\bar{\B}}$ defined on $\R^n$ ($\lambda=1, n$), whose graph is a hypersurface with constant $\s_\lambda$ curvature.
 Furthermore, the image of the Gauss map of $\M_{\z^\lambda}$ is the convex
hull of $\bar{\B}.$

The next lemma gathers properties of semitroughs that we will use to understand the asymptotic behavior of the barriers.
In the case of constant Gauss curvature, similar properties have been proved in \cite{BS}.

\begin{lemm}\label{cblem1}
Let $\bar{\B}$ be a closed ball of $\dS^{n-1}$ such that $\pi-\delta_0\geq\delta(\B)\geq\delta_0>0$ and let
$\z^\lambda_{\bar{\B}}$ be the {corresponding} semitrough with $\s_\lambda$ curvature equals $\binom{n}{\lambda},$ where $\lambda=1, n.$
Then the following hold:

(1) Let $g(x)=\sqrt{1+|x|^2},$ $\td{B}$ be the convex hull of $\B$ in $B_1,$
and $\p_0\td{B}=\p\td{B}\cap B_1.$ Then
 \be\label{cb1.3}
g\geq \z^\lambda_{\bar{\B}}>V_{\bar{\B}}=V_{\td{B}},
\ee
as $x=r\theta, r\goto\infty $ for any fixed $\theta\in\mathbb{S}^{n-1}$,
\be\label{cb1.4}
\left\{
\begin{aligned}
&\z^\lambda_{\bar{\B}}-V_{\bar{\B}}\goto \frac{l_\lambda}{\sqrt{1-\alpha^2}},\,\,
\text{when {\bf $\theta\notin\td{B}$} is perpendicular to $\p_0\td{B},$ }\\
&\z^\lambda_{\bar{\B}}-V_{\bar{\B}}\goto 0,\,\,\text{otherwise,}\\
\end{aligned}
\right.
\ee
where $-1<\alpha<1$ depends on $\delta(\bar{\B}),$
$V_{\bar{\B}}(x)=\sup\limits_{\xi\in\bar{\B}}\xi\cdot x,$ and $V_{\td{B}}(x)=\sup\limits_{\xi\in\td{B}}\xi\cdot x.$

(2) For all compact sets $K\subset \R^n$ there exists $\delta=\delta(K, \delta_0, \lambda, n)>0$ such that for all $x\in K$
\be\label{cb1.6}
\z^\lambda_{\bar{\B}}(x)\geq V_{\bar{\B}}(x)+\delta.
\ee

(3) For all compact sets $K\subset\R^n$ there exists $\upsilon_K=\upsilon(K, \delta_0, \lambda, n)\in (0, 1]$ such that
for all $x, y\in K$
\be\label{cb1.7}
|\z^\lambda_{\bar{\B}}(x)-\z^\lambda_{\bar{\B}}(y)|\leq (1-\upsilon_K)|x-y|.
\ee

(4) Let $\z^1_{\bar{\B}}$ denote the semitrough with $\s_1$ curvature equals $n$, and $\z^n_{\bar{\B}}$ denote the semitrough with
$\s_n$ curvature equals $1.$ Moreover, the blowdown of $\z^1_{\bar{\B}}$ and $\z^n_{\bar{\B}}$ are $V_{\bar{\B}}$. Then $\z^1_{\bar{\B}}>\z^n_{\bar{\B}}.$
\end{lemm}
\begin{proof}
Notice that, since $g(x)$ is invariant under the Lorentzian transform, we only need to look at these assertions for the standard semitrough. For part (1) to part (3), since the constant Gauss curvature case has been proved in \cite{BS}, we only need to consider the standard semitrough of constant mean curvature.

We first prove (1). When $\frac{x}{|x|}\in\bar{\B}_+:=\lt\{\xi \big|\,\, |\xi|=1, \xi_1\geq 0\rt\},$ by Lemma 5.1 of \cite{CT} we have
\begin{align*}
\z^1(x)-V_{\bar{\B}_+}(x)&=\sqrt{f^2(x_1)+|\bar{x}|^2}-|x|\\
&=\frac{f^2(x_1)-x_1^2}{\sqrt{f^2(x_1)+|\bar{x}|^2}+|x|}>0.
\end{align*}
When $\frac{x}{|x|}\notin\bar{\B}_+,$ we have
\[\z^1(x)-V_{\bar{\B}_+}(x)=\frac{f^2(x_1)}{\sqrt{f^2(x_1)+|\bar{x}|^2}+|\bar{x}|}>0.\]
It's easy to see that equations \eqref{cb1.3} and \eqref{cb1.4} follow through directly.

Part (2) can be derived from part (1); part (3) is due to $\z^1$ is spacelike. Thus, we only need to show part (4).
Let $\z^1(x)=\sqrt{f_1^2(x_1)+|\bx|^2},$ then by the proof of Lemma 5.1 in \cite{CT}, we know $f_1$ is the solution of
\[\frac{f_1''}{(1-f_1'^2)^{3/2}}+\frac{(n-1)}{f_1(1-f_1'^2)^{1/2}}=n.\]
Let $\z^n(x)=\sqrt{f_n(x_1)^2+|\bx|^2},$ then by Maclaurin's inequality and Section 2 of \cite{GJS} we get
\[\frac{nf_1''}{f_1^{n-1}(1-f_1'^2)^{n/2+1}}\leq 1=\frac{nf_n''}{f_n^{n-1}(1-f_n'^2)^{n/2+1}}.\]
Moreover, applying Lemma 5.1 of \cite{CT} and Section 2 of \cite{GJS}, we have that
$\lim\limits_{t\goto-\infty}(f_1(t)-f_n(t))=\lt(1-\frac{1}{n}\rt)>0,$
and $\lim\limits_{t\goto\infty}(f_1(t)-f_n(t))=0.$
By the Comparison Theorem we conclude that $f_1(t)>f_n(t)$ for all $t.$
This completes the proof of part (4).
\end{proof}

\subsection{Construction of the lower barrier.}
\label{construction lb}
Let's recall the following Lemma from \cite{BS}:
\begin{lemm}
\label{cblem2}(Lemma 4.5 in \cite{BS})
Let $\F$ be the closure of some open nonempty subset of the ideal boundary $\dS^{n-1}$ with $\T \F\in C^{1, 1}.$
There exists $\delta_0>0$ such that the following holds:

(1) $\F$ and $\bar{\F}^c$ are the union of closed balls of $\dS^{n-1}$ of radius $\delta_0.$

(2) For every $x\in\bar{\F}^c,$ there exists a closed ball $\B$ with radius bounded below by $\delta_0$ which contains $x$
and is contained in $\bar{\F}^c$ such that $d_S(x, \B^c)=d_S(x, \F).$
\end{lemm}

Now, for a given $\F,$ we fix $\delta_0>0$ as in Lemma \ref{cblem2}.
By letting
\be\label{def-lun}
\lu^{n}(x)=\sup\limits_{\bar{\B}\subset \F, \delta(\bar{\B})\geq\delta_0}\z^n_{\bar{\B}}(x),\ee
and
\be\label{def-uun}
\uu^{n}(x)=\inf\limits_{\bar{\B}\supset\F, \delta(\bar{\B})\leq\pi-\delta_0}\z^n_{\bar{\B}}(x),
\ee
where $\z^n_{\bar{\B}}(x)$ satisfies $\s_n(\ka[\z^n_{\bar{\B}}])=1.$
Bayard and Sch\"urer (see Theorem 1.2 in \cite{BS}) proved
\begin{lemm}
\label{cblem0}
Let $\F$ be the closure of some open nonempty subset of the ideal boundary $\dS^{n-1}$ with $\T \F\in C^{1, 1}.$
Then, there exists a unique, smooth, strictly convex, spacelike function $u:\R^n\goto\R$
with $\lu^n\leq u\leq\uu^n,$ such that its graph $\M_u$ satisfies
\[\s_n(\ka[\M_u])=1.\]
Moreover, $|u(x)-V_{\F}(x)|\goto 0$ as $|x|\goto\infty.$
\end{lemm}

We will use this $u$ as our lower barrier, and from now on we will denote it by $\lu.$

\subsection{Construction of the upper barrier}
\label{construction ub}
Next, we will construct the upper barrier. Let
\[\lu_1(x)=\sup\limits_{\bar{\B}\subset \F, \delta(\bar{\B})\geq\delta_0} \z^1_{\bar{\B}}(x)
\,\,\text{and}\,\,\uu_1(x)=\inf\limits_{\F\subset\bar{\B}, \delta(\bar{\B})\leq\pi-\delta_0}\z^1_{\bar{\B}}(x),\]
where $\z^1_{\bar{\B}}(x)$ are semitroughs satisfying $\s_1(\ka[\z^1_{\bar{\B}}])=n.$ Then $\lu_1(x)$ and $\uu_1(x)$
are weak sub and super solutions to the prescribed mean curvature equation
\be
\label{cu1.1}
\left\{
\begin{aligned}
&\textup{div}\lt(\frac{Du}{\sqrt{1-|Du|^2}}\rt)=n\\
&|Du(x)|<1\,\,\text{for all}\,\,x\in\R^n.
\end{aligned}
\right.
\ee

By Theorem 6.1 of \cite{CT} we know that there exists a smooth solution $h(x)$ of \eqref{cu1.1} satisfies
\[\lu_1(x)\leq h(x)\leq\uu_1(x)\,\,\text{for all}\,\, x\in\R^n\]
and $\s_1(\ka[h(x)])=n.$

We will use this $h(x)$ as our upper barrier. We note that, by Theorem 3.1 of \cite{CT} and our assumptions on $\F,$ we have $h(x)$ is strictly convex. Moreover, applying Lemma 4.4 of \cite{BS} we get,
as $|x|\goto\infty,$ $\lu(x)-V_\F(x)\goto 0.$ Thus, Lemma \ref{cblem1} yields as $|x|\goto\infty,$ $h(x)\geq \lu(x).$ By the Comparison Theorem we know that
$h(x)>\lu(x)$ for all $x\in\R^n.$

\subsection{Legendre transform of barrier functions}
\label{lt for bf}
We will denote the Legendre transform of $\lu$ and $h$ by $\lus$ and $h^*$ respectively. In this subsection, we will discuss some basic properties of $\lus$ and $h^*$ that will be used later.
\begin{lemm}
\label{lem-boundary-lus}
Let $\lu$ be the lower barrier function constructed in Subsection \ref{construction lb}, and let $\lus$ denote its Legendre transform.
Then we have
\[\lus=0\,\,\,\mbox{on $\p\td{F}.$}\]
\end{lemm}
\begin{proof}
For any $\xi\in\p\td{F},$ by the definition of Legendre transform we have
\[\lus(\xi)=\sup\limits_{x\in\R^n}\{x\cdot\xi-\lu(x)\}.\]
It's clear that
\[x\cdot\xi-\lu(x)\leq V_{\F}(x)-\lu(x).\]
Therefore by Lemma \ref{cblem1} we know that $\lus(\xi)\leq 0.$ On the other hand, there exists $\xi_0\in\dS^{n-1}$ such that
$V_\F(\xi_0)=\xi_0\cdot \xi.$ Then we get,
\[\sup\limits_{x\in\R^n}\{x\cdot\xi-\lu(x)\}\geq V_\F(r\xi_0)-\lu(r\xi_0),\,\,\mbox{for any $r>0$.}\]
Let $r\goto \infty$ we conclude $\lus(\xi)\geq 0.$ This completes the proof of the Lemma.
\end{proof}

In the next Lemma, we will compare the Legendre transform of $\lu$ and $h.$
\begin{lemm}
\label{cblem3}
Let $\lu$ be the lower barrier function constructed in Subsection \ref{construction lb}, $h$ be the upper barrier function constructed in Subsection
 \ref{construction ub}. Let $\lus$ and $h^*$ denote the Legendre transform of $\lu$ and $h$ respectively. Then we have,
$h^*(\xi)\leq\lu^*(\xi)$ for all $\xi\in\td{F}.$
\end{lemm}
\begin{proof}
For any $\xi\in\td{F},$ there exist $x, y\in\R^n$ such that
\[Dh(x)=\xi=D\lu(y),\]
Therefore,
\begin{align*}
h^*(\xi)-\lu^*(\xi)&=x\cdot \xi-h(x)-y\cdot \xi+\lu(y)\\
&<(x-y)\cdot \xi+\lu(y)-\lu(x)<0,
\end{align*}
where the last inequality comes from $\lu(x)$ is strictly convex.
\end{proof}

\bigskip
\section{Construction of the convergence sequence}
\label{cs}
Let's consider the following Dirichlet Problem
\be\label{cs1.1}
\left\{
\begin{aligned}
F(\w\gas_{ik}\us_{kl}\gas_{lj})&=\frac{1}{\binom{n}{k}^{\frac{1}{k}}}\,\,\text{in $\td{F}$}\\
\us&=0\,\,\text{on $\partial\td{F},$}
\end{aligned}
\right.
\ee
where $\w=\sqrt{1-|\xi|^2,}$ $\gas_{ij}=\delta_{ij}-\frac{\xi_i\xi_j}{1+\w},$
$F(\w\gas_{ik}\us_{kl}\gas_{lj})=\lt(\frac{\s_n}{\s_{n-k}}(\ka*[\w\gas_{ik}\us_{kl}\gas_{lj}])\rt)^{1/k},$
and $\td{F}$ is the convex hull of $\F$ in $B_1.$ Note that, by Subsection \ref{lt},
Lemma \ref{lem-boundary-lus}, Lemma \ref{cblem3}, and Maclaurin's inequality
it's easy to see that, $\lus$ is a supersolution of \eqref{cs1.1} and $h^*$ is a subsolution of \eqref{cs1.1}. Therefore, if $\us$ is a solution
of \eqref{cs1.1}, then the Legendre transform of $\us,$ denoted by $u,$ satisfies:
\[u\,\, \mbox{is defined on $D\us(\td{F})\supset D\lus(\td{F})=\R^n,$}\]
and
\[\sigma_k(\ka[\M_u])=\binom{n}{k}.\]
Moreover, we have the following Lemma.
\begin{lemm}
\label{lem-legendre-boundary}
Let $\F\subset\dS^{n-1},$ $\td{F}=\conv{\F},$ and $\us$ be a solution of
\be\left\{
\begin{aligned}
F(\w\gas_{ik}\us_{kl}\gas_{lj})&=\frac{1}{\binom{n}{k}^{\frac{1}{k}}}\,\,\text{in $\td{F}$}\\
\us&=\vp\,\,\text{on $\partial\td{F}.$}
\end{aligned}
\right.
\ee
Then, the Legendre transform of $\us$ denoted by $u$ satisfies, when $\frac{x}{|x|}\in\F$
\be\label{cv0.1}
u(x)-|x|\goto-\vp\lt(\frac{x}{|x|}\rt)\,\,\mbox{as $|x|\goto\infty,$ uniformly}.
\ee
\end{lemm}
\begin{proof}
We will show
\be\label{cv0.2}
\lim\limits_{r\goto\infty}(u(r\theta)-r)=-\vp(\theta),
\ee
and the convergence is uniform in $\F.$
If \eqref{cv0.2} is not true, then there would exist two sequences $\{r_i\}_{i=1}^\infty, \{\theta_i\in\F\}_{i=1}^{\infty},$ where $r_i\goto \infty$ as $i\goto\infty,$
and a fixed $\e_0>0,$ such that for any $i\in\mathbb{N}$ we have
\[|u(r_i\theta_i)-r_i+\vp(\theta_i)|>\e_0.\]
Since
\begin{align*}
&u(r_i\theta_i)-r_i+\vp(\theta_i)\\
&=\sup\limits_{\xi\in\td{F}}\{r_i\theta_i\cdot\xi-\us(\xi)\}-r_i+\vp(\theta_i)\\
&\geq r_i-\vp(\theta_i)-r_i+\vp(\theta_i),
\end{align*}
we get
\[u(r_i\theta_i)-r_i+\vp(\theta_i)>\e_0.\]
Therefore, for each $i\in\mathbb{N},$ there exists $\xi_i\in\td{F}$ such that
\be\label{cv0.3}
r_i\theta_i\cdot\xi_i-r_i>\e_0+\us(\xi_i)-\vp(\theta_i).
\ee
Without loss of generality, we assume $\{\theta_i\}_{i=1}^n$ converges to some $\theta_0\in\F$.
If there exists a convergent subsequence of $\{\xi_i\}_{i=1}^\infty,$ which we denote by $\{\xi_{l_j}\}_{j=1}^\infty,$
such that $\xi_{l_j}\goto\xi_0\neq\theta_0.$ Then as $j\goto\infty$ we can see that the l.h.s of \eqref{cv0.3}
goes to $-\infty,$ while the r.h.s is bounded from below. This leads to a contradiction.
Thus, we have $\lim\limits_{i\goto\infty}\xi_i=\theta_0.$ However, in this case we get as $i\goto\infty$
the l.h.s of \eqref{cv0.3} is nonpositive, while the r.h.s $\goto \e_0,$ which is a contradiction.
Therefore, Lemma \ref{lem-legendre-boundary} is proved.
\end{proof}

 By Lemma \ref{lem-legendre-boundary} we obtain, if $\us$ is a solution of \eqref{cs1.1}, then its Legendre transform $u$ also satisfies, when $\frac{x}{|x|}\in\F,$
\[u(x)\goto |x|\,\,\mbox{as $|x|\goto\infty.$}\]
From the discussion above we can see that, in order to construct an entire, strictly convex, spacelike constant $\s_k$ curvature hypersurface with prescribed lightlike directions, we only need
to show equation \eqref{cs1.1} is solvable. Unfortunately, equation \eqref{cs1.1} is a degenerate equation. Therefore, we will consider the solvability
of the following approximating problem instead.

\be\label{cs1.2}
\left\{
\begin{aligned}
F(\w\gas_{ik}\ujs_{kl}\gas_{lj})&=\frac{1}{\binom{n}{k}^{\frac{1}{k}}}\,\,\text{in $\td{F}_J$}\\
\ujs&=\vjs\,\,\text{on $\partial\td{F}_J,$}
\end{aligned}
\right.
\ee
where $\vjs=\lus|_{\partial\td{F}_J},$ and $\{\td{F}_J\}_{J\in\mathbb{N}}$ is a sequence of strictly convex set satisfying $\td{F}_J\subset\td{F}_{J+1}\subset\td{F}$ and $\p\td{F}_J$ is smooth.
In the following, we will show the existence of the solutions to equation \eqref{cs1.2}.

\subsection{$C^0$ estimates.}
\label{c0b} Since $\ujs$ is a convex function we get,
\[\sup\limits_{\td{F}_J}\ujs\leq\max\limits_{\partial\td{F}_J}\vjs.\]
Moreover, since $h^*$ is a subsolution of \eqref{cs1.1} and $h^*\leq\lus,$ by the maximum principle we conclude,
\[\ujs>\hs\,\,\mbox{in $\td{F}_J$}.\]

\subsection{$C^1$ estimates.}
\label{c1b} By Section 2 of \cite{CNS}, we know that we can always construct a subsolution
$\lujs$ such that
\be\label{c1e1.1}
\left\{
\begin{aligned}
F(\w\gas_{ik}\lujs_{kl}\gas_{lj})&\geq\frac{1}{\binom{n}{k}^{\frac{1}{k}}}\,\,\text{in $\td{F}_J$}\\
\lujs&=\vjs\,\,\text{on $\partial\td{F}_J.$}
\end{aligned}
\right.
\ee
Then, by the convexity of $\ujs$ we obtain
\[|D\ujs|\leq\max\limits_{\partial\td{F}_J}|D\lujs|.\]

\subsection{$C^2$ boundary estimates.}
\label{c2b} For our convenience, in this subsection we will use the hyperbolic model. Following the discussion in Subsection \ref{gg} we can
write equation \eqref{cs1.2} as follows:
\be\label{es1.2}
\left\{
\begin{aligned}
F(v_{ij}-v\delta_{ij})&=\frac{1}{\binom{n}{k}^{\frac{1}{k}}},\,\,\mbox{in $U_J$}\\
v&=\frac{\vjs(\xi)}{\sqrt{1-|\xi|^2}},\,\,\mbox{on $\partial U_J,$}
\end{aligned}
\right.
\ee
where $v_{ij}=\bn_j\bn_iv$ denote the covariant derivative with respect to the hyperbolic metric and
$U_J=P^{-1}(\td{F}_J)\subset\mathbb{H}^n(-1).$ Here we want to point out that $v=\frac{\ujs(\xi)}{\sqrt{1-|\xi|^2}}.$

Equation of this type has been studied by Bo Guan in \cite{Guan}. However, our function $F$ is slightly different from functions in \cite{Guan}.
More precisely, our function $F$ doesn't satisfy the assumption (1.7) in \cite{Guan}. Therefore, in order to obtain the
$C^2$ boundary estimates, we need to give
a different proof of Lemma 6.2 in \cite{Guan}, i.e., we need to construct a barrier function $\td{\ba}$ satisfying
\[\mathfrak{L}\td{\ba}:=F^{ij}\bar\nabla_{ij}\td{\ba}-\td{\ba}\sum F^{ii}\leq -c(1+\sum F^{ii})\,\,\mbox{in $U_J$},\]
and
\[\td{\ba}\geq 0\,\,\mbox{on $\partial U_J,$}\]
where $U_J=P^{-1}(\td{F}_J)\subset\mathbb{H}^n(-1),$ $\bar\nabla$ is the covariant derivative with respect to the hyperbolic metric, and $c>0.$
Note that, in \cite{RWX}, we have constructed such $\td{\ba}$ for the special case when $U_J=P^{-1}(B_{r_{J}}),$
where $B_{r_J}$ is a ball of radius $r_J<1.$ Here, the main difficulty is that $\p U_J$ doesn't lie on a plane
$\mathbb{P}:=\{x_{n+1}=c\}$ anymore. Hence, we can no longer construct $\td{\ba}$ using $c-x_{n+1}.$
In order to conquer this difficulty, we prove the following equality.
\begin{lemm}
\label{c2blem1.1}
For any function $u\in C^2(\Omega), \Omega\subseteq B_1(0),$ we have
\be\label{c2b1.1}
\bar{\nabla}_i\bar\nabla_j\lt(\frac{u}{\w}\rt)-\frac{u}{\w}\delta_{ij}=\w\gas_{ik}u_{kl}\gas_{lj},
\ee
where $\bar\nabla$ denotes the covariant derivative with respect to the hyperbolic metric on the Klein ball,
$\w=\sqrt{1-|\xi|^2},$ $\gas_{ik}=\delta_{ik}-\frac{\xi_i\xi_k}{1+\w},$ and $u_{kl}=\frac{\T^2u}{\T\xi_k\T\xi_l}.$
\end{lemm}
\begin{proof}
Let's denote $g_{ij}=\delta_{ij}-\xi_i\xi_j,$ then $g^{ij}=\delta_{ij}+\frac{\xi_i\xi_j}{1-|\xi|^2}.$ Recall Lemma 4.5 of \cite{CT}, we know that
the hyperbolic metric on the Klein ball is
\[k_{ij}=\frac{1}{1-|\xi|^2}\lt(\delta_{ij}+\frac{\xi_i\xi_j}{1-|\xi|^2}\rt)=\frac{g^{ij}}{{\w}^2}.\]
Now, for any function $u$ defined on $\Omega\subseteq B_1,$ let $\td{u}=\frac{u}{\w}.$ Also note that
$\gas_{ij}$ is the square root of the matrix $(g_{ij}),$ i.e., $g_{ij}=\gas_{im}\gas_{mj}.$
We define a new frame $\{e_1, \cdots, e_n\}$ by
\[e_i=\w\gas_{ik}\frac{\T}{\T\xi_k},\,\,\mbox{for $1\leq i\leq n$.}\]
It's clear that
\[k(e_i, e_j)={\w}^2\gas_{im}k_{mn}\gas_{nj}=\delta_{ij}.\]
Hence, $\{e_1, \cdots, e_n\}$ is an orthonormal frame with respect to the metric $k.$
Let's calculate the Christoffel symbol $\Gamma_{mn}^s$. Recall that
\be\label{c2b1.2}
\Gamma_{mn}^s=\frac{1}{2}k^{sl}\left(\frac{\T k_{ml}}{\T \xi_n}+\frac{\T k_{nl}}{\T \xi_m}-\frac{\T k_{mn}}{\T \xi_l}\right).
\ee
A straightforward calculation yields
\be\label{c2b1.3}
\frac{\T k_{ml}}{\T \xi_n}= 2\frac{\xi_ng^{ml}}{{\w}^4}+\frac{1}{{\w}^2}
\left(\frac{\xi_l\delta_{mn}+\xi_m\delta_{ln}}{{\w}^2}+\frac{2\xi_l\xi_m\xi_n}{{\w}^4}\right).
\ee
Combining \eqref{c2b1.3} with \eqref{c2b1.2}, we obtain
\begin{eqnarray}
\Gamma_{mn}^s=\frac{1}{{\w}^2}\left(\xi_m\delta_{ns}+\xi_n\delta_{ms}\right).
\end{eqnarray}
Moreover, it's easy to see that
\begin{eqnarray}
\frac{\T \tilde{u}}{\T \xi_m}&=&\frac{1}{\w}\frac{\T u}{\T \xi_m}+\frac{\xi_m u}{{\w}^3}\\
\frac{\T^2 \tilde{u}}{\T \xi_m\T \xi_n}&=&\frac{1}{\w}\frac{\T^2 u}{\T\xi_m\T \xi_n}+\frac{\xi_m}{{\w}^3}\frac{\T u}{\T \xi_n}
+\frac{\xi_n}{{\w}^3}\frac{\T u}{\T \xi_m}+\frac{\delta_{mn} u}{{\w}^3}+3\frac{\xi_m\xi_n}{{\w}^5}u.
\end{eqnarray}
Therefore we have,
\be\label{c2b1.4}
\begin{aligned}
\bar\nabla_i\bar\nabla_j\td{u}&={\w}^2\gas_{im}\bar\nabla_{{\T}_m}\bar\nabla_{{\T}_n}\td{u}\gas_{nj}\\
&={\w}^2\gas_{im}\lt(\frac{{\T}^2\td{u}}{\T\xi_m\T\xi_n}-\Gamma^s_{mn}\frac{\T\td{u}}{\T\xi_s}\rt)\gas_{jn}\\
&={\w}^2\gas_{im}\lt[\frac{u_{mn}}{\w}+\frac{\xi_mu_n}{{\w}^3}+\frac{u_m\xi_n}{{\w}^3}+\frac{\delta_{mn}u}{{\w}^3}\right.\\
&+\frac{3\xi_m\xi_nu}{{\w}^5}-\lt(\frac{\xi_n}{{\w}^2}\delta_{sm}+\frac{\xi_m}{{\w}^2}\delta_{ns}\rt)
\left.\lt(\frac{u_s}{\w}+\frac{\xi_su}{{\w}^3}\rt)\rt]\gas_{nj}\\
&=\w\gas_{im}u_{mn}\gas_{nj}+\xi_iu_n\gas_{nj}+\xi_ju_m\gas_{mi}+\frac{ug_{ij}}{\w}+\frac{3\xi_i\xi_j}{\w}u\\
&-{\w}^2\gas_{im}
\lt(\frac{u_m\xi_n}{{\w}^3}+\frac{\xi_n\xi_m}{{\w}^5}u+\frac{\xi_mu_n}{{\w}^3}+\frac{\xi_m\xi_n}{{\w}^5}u\rt)\gas_{nj}\\
&=\w\gas_{im}u_{mn}\gas_{nj}+\frac{u}{\w}\delta_{ij},
\end{aligned}
\ee
where we have used $\gas_{ik}\xi_k=\w\xi_i.$ This completes the proof of Lemma \ref{c2blem1.1}.
\end{proof}

Recall equation \eqref{cs1.2}
\[F\lt(\w\gas_{ik}\ujs_{kl}\gas_{lj}\rt)=\frac{1}{\binom{n}{k}^\frac{1}{k}}.\]
We denote
\[a^*_{kl}=\w\gas_{ki}\ujs_{ij}\gas_{jl},\]
then
\[G^{ij}=\frac{\T F}{\T a^*_{kl}}\frac{\T a^*_{kl}}{\ujs_{ij}}=\w\gas_{ik}F^{kl}\gas_{lj}.\]
It's easy to see that
\[L\ujs=\frac{1}{\binom{n}{k}^\frac{1}{k}},\]
where $L:=G^{ij}\T_{\xi_i}\T_{\xi_j}.$

\begin{lemm}
\label{c2blem1.2}
For any constant $a>0$, there exist positive constants $t,N>0$ large, and $\delta$ sufficiently small, such that the function
$\ba=\ujs-\lujs+td-Nd^2$ satisfies
\[L\ba\leq- a \sum G^{ii}\,\,\mbox{in $\td{F}_J\cap B_\delta,$}\]
and
\[\ba\geq 0\,\, \mbox{on $\T(\td{F}_J\cap B_\delta).$}\]
Here, $d$ is the Euclidean distance function to $\T\td{F}_J$, $B_\delta$ is a ball of radius $\delta$ centered at a point on $\T\td{F}_J$,
and $\lujs$ is the subsolution to \eqref{cs1.2} constructed in Subsection \ref{c1b}.
\end{lemm}
\begin{proof}
By the convexity of $\T\td{F}_J$, in a small neighborhood of $\T\td{F}_J$ we have (see \cite{GT})
\[\kappa[D^2 d]=\lt[\frac{-\ka_1}{1-\ka_1d}, \frac{-\ka_2}{1-\ka_2d}, \cdots, \frac{-\ka_{n-1}}{1-\ka_{n-1}d}, 0\rt],\]
where $\ka_i>0,\,\,1\leq i\leq n-1,$ are the principal curvatures of $\p \tilde{F}_J$.
Therefore,
\[\ka[tD^2d-ND^2d^2]=\lt[\frac{-\ka_1}{1-\ka_1d}(t-2Nd), \cdots, \frac{-\ka_{n-1}}{1-\ka_{n-1}d}(t-2Nd), -2N\rt],\]
which implies
$$tD^2d-ND^2d^2\leq- C_0t I_n,$$
where we choose $\delta>0$ small such that $\delta N<\frac{t}{4}$, $\kappa_i\delta<1/2$, $2N>C_0t$, $C_0>0$ depends on $\T\td{F}_J,$ and $I_n$ is the $n$ dimensional identity matrix.
Moreover, since $\lujs$ is a subsolution of \eqref{cs1.2}, we know that
\[L\lujs\geq\frac{1}{\binom{n}{k}^{\frac{1}{k}}} \,\,\mbox{in $\td{F}_J\cap B_\delta.$}\]
It's clear that when $N=N(\td{F}_J, a, \delta), t=t(\td{F}_J, a)>0$ large, we have
\[L\ba\leq-C_0t\sum G^{ii}\leq- a\sum G^{ii}.\]
\end{proof}
From Lemma \ref{c2blem1.1} and Lemma \ref{c2blem1.2} we conclude
\begin{lemm}
For any constant $c>0$, there exist positive constants $t,N>0$ large, and $\delta$ sufficiently small, such that the function
\[\td{\ba}=\frac{\ba}{\w}=\frac{\ujs-\lujs+td-Nd^2}{\w}\]
satisfies
\be\label{c2b1.5}
\mathfrak{L}\td{\ba}:=F^{ij}\bar\nabla_{ij}\td{\ba}-\td{\ba}\sum F^{ii}\leq -c(1+\sum F^{ii})\,\,\mbox{in $U_{J\delta}$},
\ee
and
\be\label{c2b1.6}
\td{\ba}\geq 0\,\,\mbox{on $\partial U_{J\delta},$}
\ee
where $U_{J\delta}=P^{-1}(\td{F}_J\cap B_\delta)\subset\mathbb{H}^n(-1).$
\end{lemm}
The rest of $C^2$ boundary estimates follows from \cite{Guan} directly.

\subsection{Global $C^2$ estimates.}
\label{c2g}
In this subsection, we will still use the hyperbolic model and study the equation \eqref{es1.2}. We will estimate
$|\bn^2 v|$ on $\bar{U}_J.$ Keep in mind that a bound on $|\bn^2 v|$ yields a bound on $|\p^2\ujs|.$
\begin{lemm}
\label{c2glem2}
Let $v$ be the solution of \eqref{es1.2}. Denote the eigenvalues of $(v_{ij}-v\delta_{ij})$ by $\lambda[v_{ij}-v\delta_{ij}]=(\lambda_1, \cdots, \lambda_n).$  Then, $|\lambda[v_{ij}-v\delta_{ij}]|$ is bounded from above.
\end{lemm}
\begin{proof}
In this proof we will denote $\Lambda_{ij}=v_{ij}-v\delta_{ij},$ where $v_{ij}$ is the covariant derivatives with respect to the hyperbolic metric. We will use $\lambda=(\lambda_1,\lambda_2,\cdots,\lambda_n)$ to denote the eigenvalues of the matrix $\Lambda$. From Subsection \ref{gg} and \ref{lt} we know that $\lambda=\ka^*=\ka^{-1}.$

Let's recall the following geometric formulae:
\be\label{comm}
\begin{aligned}
\Lambda_{ijk}&=\Lambda_{ikj}\\
\Lambda_{lkji}-\Lambda_{lkij}&=v_{lkji}-v_{lkij}\\
&=-v_{lj}\delta_{ik}+v_{li}\delta_{jk}-v_{jk}\delta_{il}+v_{ik}\delta_{jl}.
\end{aligned}
\ee
Set
$$M=\max\limits_{P\in \overline{U}_J}\max\limits_{|\xi|=1, \xi\in T_{P}\mathbb{H}^n}\left(\log \Lambda_{\xi\xi}+Nx_{n+1}\right),$$
where $N$ is a constant to be determined later and $x_{n+1}$ is the coordinate function. By the discussion in Subsection \ref{c2b} we already know that
$|\lambda|$ is bounded on $\p U_J.$ Therefore, in the following, we may assume $M$ is achieved at $P_0\in U_J$ for some direction $\xi_0.$ Choosing a orthonormal frame $\{\tau_1, \cdots, \tau_n\}$ around $P_0$ such that $\tau_1(P_0)=\xi_0$ and $\Lambda_{ij}(P_0)=\lambda_i\delta_{ij}.$

Now, let's consider the test function
\[\phi=\log\Lambda_{11}+Nx_{n+1}.\]
At its maximum point $P_0$, we have
\begin{eqnarray}\label{3.6}
0=\phi_i&=&\frac{\Lambda_{11i}}{\Lambda_{11}}+N(x_{n+1})_i\\
0\geq\phi_{ii}&=&\frac{\Lambda_{11ii}}{\Lambda_{11}}-\frac{\Lambda_{11i}^2}{\Lambda_{11}^2}+N(x_{n+1})_{ii}.
\end{eqnarray}
Using $(x_{n+1})_{ij}=x_{n+1}\delta_{ij}$, we get
\begin{eqnarray}\label{3.8}
0\geq F^{ii}\phi_{ii}=\frac{F^{ii}\Lambda_{11ii}}{\Lambda_{11}}-\frac{F^{ii}\Lambda_{11i}^2}{\Lambda_{11}^2}+Nx_{n+1}\sum_i F^{ii}.
\end{eqnarray}
Applying \eqref{comm}, we obtain
\begin{eqnarray}%\label{3.9}
\Lambda_{11ii}=\Lambda_{i11i}=\Lambda_{i1i1}+v_{ii}-v_{11}=\Lambda_{ii11}+\Lambda_{ii}-\Lambda_{11}.\nonumber
\end{eqnarray}
Thus, we have
\begin{eqnarray}\label{3.10}
F^{ii}\Lambda_{11ii}=F^{ii}\Lambda_{ii11}+F^{ii}\Lambda_{ii}-\Lambda_{11}\sum_{i}F^{ii}.
\end{eqnarray}
Differentiating equation \eqref{es1.2} twice we get
\begin{eqnarray}\label{3.11}
F^{ii}\Lambda_{ii11}&=&-F^{pq,rs}\Lambda_{pq1}\Lambda_{rs1}\\
&=&-F^{pp,qq}\Lambda_{pp1}\Lambda_{qq1}-\sum_{p\neq q}\frac{F^{pp}-F^{qq}}{\lambda_p-\lambda_q}\Lambda_{pq1}^2,\nonumber
\end{eqnarray}
here the second equality comes from Theorem 5.5 of \cite{Bal}.
Since $(\sigma_n/\sigma_{n-k})^{1/k}$ is concave, the first term of \eqref{3.11} is nonnegative.
Combing \eqref{3.8}-\eqref{3.11}, we obtain at $P_0$,
\begin{eqnarray}\label{3.12}
0\geq F^{ii}\phi_{ii}&\geq &-\frac{1}{\Lambda_{11}}\sum_{p\neq q}\frac{F^{pp}-F^{qq}}{\lambda_p-\lambda_q}\Lambda_{pq1}^2-\frac{F^{ii}\Lambda_{11i}^2}{\Lambda_{11}^2}+(Nx_{n+1}-1)\sum_i F^{ii}\\
&\geq&\frac{1}{\Lambda_{11}}\sum_{i\neq 1}\frac{F^{ii}-F^{11}}{\lambda_1-\lambda_i}\Lambda_{11i}^2-\frac{F^{ii}\Lambda_{11i}^2}{\Lambda_{11}^2}+(Nx_{n+1}-1)\sum_i F^{ii}.\nonumber
\end{eqnarray}
In order to analyze the right hand side of inequality \eqref{3.12}, we need an explicit expression of $F^{ii}$.
By a straightforward calculation we have,
\begin{eqnarray}\label{FF}
kF^{k-1}F^{ii}=\frac{\sigma_n^{ii}\sigma_{n-k}-\sigma_n\sigma_{n-k}^{ii}}{\sigma_{n-k}^2}.
\end{eqnarray}
Note that
\begin{eqnarray}
&&\sigma_n^{ii}\sigma_{n-k}-\sigma_n\sigma_{n-k}^{ii}\nonumber\\
&=&\sigma_{n-1}(\lambda|i)(\lambda_i\sigma_{n-k-1}(\lambda|i)+\sigma_{n-k}(\lambda|i))
-\lambda_i\sigma_{n-1}(\lambda|i)\sigma_{n-k-1}(\lambda|i)\nonumber\\
&=&\sigma_{n-1}(\lambda|i)\sigma_{n-k}(\lambda|i)\nonumber.
\end{eqnarray}
Here and in the following, $\sigma_l(\kappa|a)$ and $\s_l(\ka|ab)$ are the $l$-th elementary
symmetric polynomials of $\kappa_1,\kappa_2,\cdots,\kappa_n$ with
$\kappa_a=0$ and $\ka_a=\ka_b=0,$ respectively. Therefore, we get
\begin{eqnarray}\label{F}
kF^{k-1}F^{ii}=\frac{\sigma_{n-1}(\lambda|i)\sigma_{n-k}(\lambda|i)}{\sigma_{n-k}^2}.
\end{eqnarray}
This implies
\begin{eqnarray}
kF^{k-1}\lt(F^{ii}-F^{11}\rt)&=&\frac{1}{\sigma_{n-k}^2}[\sigma_{n-1}(\lambda|i)\sigma_{n-k}(\lambda|i)-\sigma_{n-1}(\lambda|1)\sigma_{n-k}(\lambda|1)]\\
&=&\frac{\sigma_{n-2}(\lambda|1i)}{\sigma_{n-k}^2}[\lambda_1\sigma_{n-k}(\lambda|i)-\lambda_i\sigma_{n-k}(\lambda|1)]\nonumber\\
&=&\frac{\sigma_{n-2}(\lambda|1i)(\lambda_1-\lambda_i)}{\sigma_{n-k}^2}[(\lambda_1+\lambda_i)\sigma_{n-k-1}(\lambda|1i)+\sigma_{n-k}(\lambda|1i)]\nonumber.
\end{eqnarray}
Thus, for $i\geq 2$, we have
\begin{eqnarray}\label{3.14}
&&kF^{k-1}\lt(\frac{F^{ii}-F^{11}}{\lambda_1-\lambda_i}-\frac{F^{ii}}{\lambda_1}\rt)\\
&=&\frac{\sigma_{n-2}(\lambda|1i)}{\sigma_{n-k}^2}[(\lambda_1+\lambda_i)\sigma_{n-k-1}(\lambda|1i)+\sigma_{n-k}(\lambda|1i)-\sigma_{n-k}(\lambda|i)]\nonumber\\
&=&\frac{\sigma_{n-2}(\lambda|1i)}{\sigma_{n-k}^2}\lambda_i\sigma_{n-k-1}(\lambda|1i)\nonumber\\
&=&\frac{\sigma_{n-1}(\lambda|1)}{\sigma_{n-k}^2}\sigma_{n-k-1}(\lambda|1i)\nonumber\\
&>&0.\nonumber
\end{eqnarray}
Plugging \eqref{3.6} and \eqref{3.14} into \eqref{3.12} we conclude,
\begin{eqnarray}\label{F-inequality}
0\geq F^{ii}\phi_{ii}&\geq& -F^{11}\frac{\Lambda_{111}^2}{\Lambda_{11}^2}+(Nx_{n+1}-1)\sum_iF^{ii}\\
&=&-F^{11}N^2(x_{n+1})_1^2+(Nx_{n+1}-1)\sum_iF^{ii}.\nonumber
\end{eqnarray}
Notice that from \eqref{FF} we can see that,
$$kF^{k-1}F^{11}\leq \frac{\sigma_{n}^{11}\sigma_{n-k}}{\sigma_{n-k}^2}=\frac{1}{\Lambda_{11}\binom{n}{k}}.$$

Moreover, since $F$ is concave and homogenous of degree one we can derive
$$\sum_{i}F^{ii}\geq \binom{n}{k}^{-\frac{1}{k}}.$$
Now, by letting $N=2$ in \eqref{F-inequality} we obtain that if $M$ is achieved at an interior point $P_0\in U_J,$ then at this point $\lambda_1$ is bounded from above. Therefore, we showed that $M$ is bounded from above which in turn gives
an upper bound for $|\bn^2 v|.$
\end{proof}
Combining the results in Subsection \ref{c0b}, \ref{c1b}, \ref{c2b}, and \ref{c2g}, we conclude that the
approximating Dirichlet problem \eqref{cs1.2} is solvable.

\bigskip
\section{Convergence of solutions to a entire constant $\sigma_k$ curvature hypersurface}
\label{cv}
Let $\uj$ be the Legendre transform of $\ujs,$ where $\ujs$ is the solution of \eqref{cs1.2}.
We want to show there exists a subsequence of $\{\uj\}$ that
converges to the desired entire solution $u$ of \eqref{int1.0}.

\subsection{Local $C^0$ estimates}
\label{lc0}
Recall that Lemma \ref{cblem3} tells us
\be\label{lc01.1}
\hs(\xi)<\lus\,\,\text{in $\overline{\td{F}}_J$}.
\ee
Now we will show
\begin{lemm}
\label{lc0lem1}
$\uj<h$ in $\Omega_J:=D\ujs(\td{F}_J)\subset\R^n.$
\end{lemm}
\begin{proof} For any $x\in \Omega_J$, we suppose
$$x=D\ujs(\xi)=Dh^*(\eta).$$ Then, we have
$$u^J(x)-h(x)=x\cdot \xi-\ujs(\xi)-x\cdot\eta+h^*(\eta)<x\cdot(\xi-\eta)+h^*(\eta)-h^*(\xi)<0,$$
where the last inequality comes from the strict convexity of $h^*.$
\end{proof}
Similarly we can show
\begin{lemm}
\label{lc0lem2}
$\uj>\lu$ in $B_{J-1}(0)\subset D\lus(\td{F}_J),$ where $B_{J-1}=\{x\in\R^n| |x|<J-1\}.$ Note that, here without loss of generality, we can always choose
$\td{F}_J$ such that $B_{J-1}(0)\subset D\lus(\td{F}_J).$
\end{lemm}

Applying Lemma \ref{lc0lem1} and Lemma \ref{lc0lem2}, we conclude that
\[\lu<\uj<h\,\,\text{in $B_{J-1}(0)$.}\]

\subsection{Local $C^1$ estimates}
\label{lg}
In this subsection we will prove the local $C^1$ estimates. We will need the following lemma which was proved in Section 5 of \cite{BS}.
\begin{lemm}
\label{lc1lem1}(Lemma 5.1 in \cite{BS})
Let $\Omega\subset \R^n$ be a bounded open set. Let $u, \bar{u}, \psi:\Omega\goto\R$ be strictly spacelike.
Assume that near $\partial\Omega,$ we have $\psi>\bar{u}$ and everywhere in $\Omega$ $u\leq\bar{u}.$ We also assume $u$ is convex.
Consider the set, where $u>\psi.$ For every $x$ in that set, we get the following gradient estimate for $u:$
\[\frac{1}{\sqrt{1-|Du|^2}}\leq\frac{1}{u(x)-\psi(x)}\cdot\sup\limits_{\{u>\psi\}}\frac{\bar{u}-\psi}{\sqrt{1-|D\psi|^2}}.\]
\end{lemm}
From this Lemma we can see that, in order to prove the local $C^1$ estimates, we only need to construct a suitable spacelike function
$\psi.$ We will complete this task in the rest of this subsection.
\begin{lemm}
\label{lc1lem2}
Let $\z(x)=\sqrt{f^2(x_1)+|\bar{x}|^2},$ $\bar{x}=(x_2, \cdots, x_n),$ be the standard semitrough that satisfies
$\sigma_1(\ka[z(x)])=n.$ We have
\be\label{lc1.1}
\z(x)\leq V_{\bar{\B}_+}(x)+\frac{1}{\sqrt{1+|\bar{x}|^2}}\lt(1-V_{\bar{\B}_+}\rt)\lt(\frac{x}{|x|}\rt)
\ee
as $|x|\goto\infty,$ where $\bar{\B}_+:=\{\xi| |\xi|= 1\,\,\mbox{and $\xi_1\geq 0$}\}$ and $V_{\bar{\B}_+}(x)=\sup\limits_{\xi\in\bar{\B}_+}\xi\cdot x.$
Indeed, the inequality is uniform, i.e., for any $\e>0,$
there exists an $R_\e>0$ large, such that when $R>R_\e,$
\be\label{lc1.1*}
\z(x)< V_{\bar{\B}_+}(x)+\frac{1}{\sqrt{1+|\bar{x}|^2}}\lt(1-V_{\bar{\B}_+}\rt)\lt(\frac{x}{|x|}\rt)+\e.
\ee

\end{lemm}
\begin{proof}
For our convenience, we will prove \eqref{lc1.1} for $n=2.$ It's easy to see that when $n>2$ the proof is the same.
We also want to point out that we will apply Lemma 5.1 in \cite{CT} throughout the proof.

Case 1. When $x_1\geq 0,$ a straightforward calculation yields
\be\label{lc1.2}
\begin{aligned}
\z(x)-V_{\bar{\B}_+}(x)&=\sqrt{f^2(x_1)+x_2^2}-\sqrt{x_1^2+x_2^2}\\
&=\frac{f^2(x_1)-x_1^2}{\sqrt{f^2(x_1)+|x_2|^2}+\sqrt{x_1^2+x_2^2}}\goto 0\,\,\mbox{as $|x|\goto\infty.$}\\
\end{aligned}
\ee
Case 2. When $x_1<0,$ by a direct calculation we have
\be\label{lc1.3}
\z(x)-V_{\bar{\B}_+}(x)=\frac{f^2(x_1)}{\sqrt{f^2(x_1)+x_2^2}+|x_2|}.
\ee
We'll discuss \eqref{lc1.3} in three cases.\\
i). As $x_1\goto-\infty,$ $x_2$ is bounded, it's easy to see that
\[\frac{1}{\sqrt{1+x_2^2}}\geq\frac{l_1^2}{\sqrt{l_1^2+x_2^2}+|x_2|},\,\,\mbox{where $l_1=\frac{n-1}{n}$}.\]
ii). When $x_1$ is bounded $|x_2|\goto\infty,$ we have both
$\z(x)-V_{\bar{\B}_+}(x)$ and $\frac{1}{\sqrt{1+x_2^2}}\lt(1-V_{\bar{\B}_+}\lt(\frac{x}{|x|}\rt)\rt)$
go to $0.$\\
iii). When $|x_1|$, $|x_2|$ $\goto \infty,$ we again get both
$\z(x)-V_{\bar{\B}_+}(x)$ and $\frac{1}{\sqrt{1+x_2^2}}\lt(1-V_{\bar{\B}_+}\lt(\frac{x}{|x|}\rt)\rt)$
go to $0.$\\
It's easy to see that \eqref{lc1.1*} follows from \eqref{lc1.2} and \eqref{lc1.3}. Therefore, the Lemma is proved.
\end{proof}

Next, we will construct our spacelike function $\psi.$
\begin{lemm}
\label{lc1lem3}
Let $A_0=A_0(\lambda),$ $B_0=B_0(\lambda)$ be large numbers depending on $\lambda\in(0, 1].$ Then when $R_0>A_0,$ $R_1>B_0R_0,$
\be\label{cutoff-function}
\psi=\left\{
\begin{aligned}
&\sqrt{\lambda^2+V^2_{\bar{\B}_+}(x)}+\frac{1}{\sqrt{1+|\bar{x}|^2}}\lt(1-V_{\bar{\B}_+}\lt(\frac{x}{|x|}\rt)\rt),\,\, |x|\geq R_1\\
&\sqrt{\lambda^2+V^2_{\bar{\B}_+}(x)}+\frac{1}{\sqrt{1+|\bar{x}|^2}}\lt(1-V_{\bar{\B}_+}\lt(\frac{x}{|x|}\rt)\rt)\eta(x),\,\, R_2<|x|<R_1\\
&\sqrt{\lambda^2+V^2_{\bar{\B}_+}(x)},\,\,|x|\leq R_0
\end{aligned}
\right.
\ee
is spacelike on $\R^n,$ where $\eta(x)=\frac{|x|-R_0}{R_1-R_0}.$
\end{lemm}
\begin{proof}
 For any given point $x=(x_1,\bar{x})$, we rotate the coordinate such that, $\bar{x}=(x_2,0,0,\cdots,0)$ with $x_2\geq 0$. Thus, we have
 $|D_{\bar{x}}\psi|=\frac{\partial \psi}{\partial x_2}$.
For our convenience, we will prove this Lemma for $n=2.$ When $n>2$ the proof is the same.
In this proof, we will denote $\vp(x)=\sqrt{\lambda^2+V^2_{\bar{\B}_+}(x)},$ $g(x)=\frac{1}{\sqrt{1+x_2^2}},$
and $V(x)=V_{\bar{\B}_+}\lt(\frac{x}{|x|}\rt).$

It's easy to see that when $x_1\geq 0,$ $\psi(x)=\sqrt{\lambda^2+|x|^2},$ which is obviously spacelike. Thus, in the following, we only need to look at the case
when $x_1<0.$ Without loss of generality, we also assume $x_2\geq 0,$ then we have
\[\vp(x)=\sqrt{\lambda^2+x_2^2}\,\,\mbox{and } V(x)=\frac{x_2}{|x|}.\]

First, let's look at the region $\lt\{x\in\R^n|\,|x|\geq R_1\rt\}\cap\{x\in\R^n|x_1<0, x_2\geq 0\}.$ In this region
\[\psi(x)=\vp(x)+g(x)(1-V(x)).\]
A straightforward calculation gives
\[\vp_1=0,\,g_1=0,\,V_1=-\frac{x_1x_2}{|x|^3},\]
\[\vp_2=\frac{x_2}{\sqrt{\lambda^2+x_2^2}},\, g_2=-g^3x_2,\,\mbox{and } V_2=\frac{x_1^2}{|x|^3}.\]
Thus,
\[\psi_1=\vp_1+g_1(1-V)-gV_1=-gV_1,\]
and
\[\psi_2=\vp_2+g_2(1-V)-gV_2.\]
This yields,
\begin{align*}
|D\psi|^2&=g^2|DV|^2+|D\vp|^2+(1-V)^2g_2^2\\
&+2\vp_2g_2(1-V)-2g\vp_2V_2-2g(1-V)g_2V_2.\\
\end{align*}
Therefore, we get
\be\label{lc1.4}
\begin{aligned}
1-|D\psi|^2&=\frac{\lambda^2}{\lambda^2+x_2^2}-(1-V)^2g^6x_2^2\\
&-g^2\frac{x_1^2}{|x|^4}-2\frac{x_2}{\sqrt{\lambda^2+x_2^2}}(-g^3x_2)(1-V)\\
&+2g\frac{x_2}{\sqrt{\lambda^2+x_2^2}}\frac{x_1^2}{|x|^3}+2\frac{g(1-V)(-g^3x_2)x_1^2}{|x|^3}.\\
\end{aligned}
\ee
In order to simplify \eqref{lc1.4} we will use the polar coordinates and let $x=(x_1, x_2)=(r\cos\theta, r\sin\theta).$
Then we have,
\be\label{lc1.5}
\begin{aligned}
\frac{1-|D\psi|^2}{g^2}&\geq\lambda^2-\frac{(1-\sin\theta)^2r^2\sin^2\theta}{(1+r^2\sin^2\theta)^2}\\
&-\frac{\cos^2\theta}{r^2}+\frac{2x_2^2(1-\sin\theta)}{\sqrt{(\lambda^2+x_2^2)(1+x_2^2)}}\\
&+\frac{2\sin\theta\cos^2\theta}{\sqrt{(\lambda^2+x_2^2)(1+x_2^2)}}-\frac{2\sin\theta\cos^2\theta(1-\sin\theta)}{1+x_2^2}\\
&=\textcircled{1}-\textcircled{2}-\textcircled{3}+\textcircled{4}+\textcircled{5}-\textcircled{6}.\\
\end{aligned}
\ee
It's easy to see that
\[\textcircled{5}-\textcircled{6}\geq 0,\]
\[\textcircled{4}-\textcircled{2}\geq\frac{x_2^2(1-\sin\theta)}{1+x_2^2}\lt[2-\frac{(1-\sin\theta)}{1+x_2^2}\rt]\geq 0,\]
and when $R_1>\frac{1}{\lambda}$
\[\textcircled{1}-\textcircled{3}>0.\]
This implies that when $R_1>\frac{1}{\lambda}$, $\psi$ is spacelike in the region $|x|\geq R_1.$

Next, let's look at the region $\lt\{x\in\R^n|\,R_0<|x|< R_1\rt\}\cap\{x\in\R^n|x_1<0, x_2\geq 0\}.$
In this region, $\psi=\vp+g(1-V)\eta.$ Differentiating it we get
\[\psi_1=-g\eta V_1+g(1-V)\eta_1,\]
and
\[\psi_2=\vp_2+g_2(1-V)\eta-g\eta V_2+g(1-V)\eta_2.\]
Thus,
\begin{align*}
|D\psi|^2&=g^2\eta^2|DV|^2+g^2(1-V)^2|D\eta|^2\\
&+|D\vp|^2+(1-V)^2\eta^2|Dg|^2-2g^2(1-V)\eta V_1\eta_1\\
&+2(1-V)\eta\vp_2g_2-2g\eta\vp_2V_2+2g(1-V)\vp_2\eta_2\\
&-2g(1-V)\eta^2g_2V_2+2g(1-V)^2\eta g_2\eta_2-2g^2\eta(1-V)V_2\eta_2.
\end{align*}
Since $\eta_1=\frac{\cos\theta}{R_1-R_0}$ and $\eta_2=\frac{\sin\theta}{R_1-R_0},$
we have
\[V_1\eta_1+V_2\eta_2=-\frac{x_1x_2}{|x|^3}\frac{1}{R_1-R_0}\frac{x_1}{|x|}
+\frac{x_1^2}{|x|^3}\frac{1}{R_1-R_0}\frac{x_2}{|x|}=0,\] this yields
\be\label{lc1.6}
\begin{aligned}
1-|D\psi|^2&=\frac{\lambda^2}{\lambda^2+x_2^2}-g^2\eta^2|DV|^2-g^2(1-V)^2|D\eta|^2\\
&-(1-V)^2\eta^2|Dg|^2-2(1-V)\eta\vp_2g_2+2g\eta\vp_2V_2\\
&-2g(1-V)\vp_2\eta_2+2g(1-V)\eta^2g_2V_2-2g(1-V)^2\eta g_2\eta_2\\
&=\frac{\lambda^2}{\lambda^2+x_2^2}-g^2\eta^2\frac{\cos^2\theta}{r^2}-\frac{g^2(1-\sin\theta)^2}{(R_1-R_0)^2}\\
&-(1-\sin\theta)^2\eta^2g^6x_2^2+2(1-\sin\theta)\eta\frac{x_2}{\sqrt{\lambda^2+x_2^2}}g^3x_2\\
&+2g\eta\frac{x_2}{\sqrt{\lambda^2+x_2^2}}\frac{x_1^2}{|x|^3}-2g(1-\sin\theta)\frac{x_2}{\sqrt{\lambda^2+x_2^2}}\frac{1}{R_1-R_0}\frac{x_2}{|x|}\\
&-2g(1-\sin\theta)\eta^2g^3x_2\frac{x_1^2}{|x|^3}+2g(1-\sin\theta)^2\eta g^3x_2\frac{1}{R_1-R_0}\frac{x_2}{|x|}\\
&=\textcircled{1}-\textcircled{2}-\textcircled{3}-\textcircled{4}+\textcircled{5}+\textcircled{6}
-\textcircled{7}-\textcircled{8}+\textcircled{9}.\\
\end{aligned}
\ee
We will divide it into two cases.\\
Case 1. When $x_2\leq R_0,$ by a careful calculation we obtain,
\[\frac{\textcircled{5}}{2}-\textcircled{4}\geq(1-\sin\theta)\eta x_2^2g^4[1-(1-\sin\theta)\eta g^2]\geq 0,\]
\begin{align*}
&\textcircled{6}-\textcircled{8}\\
&\geq 2g^2\eta\sin\theta\cos^2\theta-2g^4(1-\sin\theta)\eta^2\sin\theta\cos^2\theta\\
&=2g^2\eta\sin\theta\cos^2\theta[1-g^2(1-\sin\theta)\eta]\geq 0,\\
\end{align*}
and
\begin{align*}
&\frac{\textcircled{1}-\textcircled{2}-\textcircled{3}-\textcircled{7}}{g^2}\\
&\geq\frac{(1+x_2^2)\lambda^2}{\lambda^2+x_2^2}-\frac{1}{r^2}-\frac{1}{(R_1-R_0)^2}
-\frac{2(1-\sin\theta)r\sin^2\theta\sqrt{1+x_2^2}}{\sqrt{\lambda^2+x_2^2}(R_1-R_0)}\\
&\geq\frac{\lambda^2}{2}-\frac{1}{r^2}-\frac{1}{(R_1-R_0)^2}
+\frac{\sqrt{1+x_2^2}}{\sqrt{\lambda^2+x_2^2}}\lt(\frac{\lambda^2}{2}-\frac{2(1-\sin\theta)x_2\sin\theta}{R_1-R_0}\rt)\\
&\geq\frac{\lambda^2}{2}-\frac{1}{R_0^2}-\frac{1}{(R_1-R_0)^2}
+\frac{\sqrt{1+x_2^2}}{\sqrt{\lambda^2+x_2^2}}\lt(\frac{\lambda^2}{2}-\frac{2R_0}{R_1-R_0}\rt).\\
\end{align*}
Therefore, when $R_0>\frac{10}{\lambda},$ $R_1>R_0+\frac{10}{\lambda^2}R_0,$ we get
$1-|D\psi|^2>0$ in this case.\\
Case 2. When $x_2>R_0,$ we will group our terms differently. First, notice that
\begin{align*}
\frac{\textcircled{5}}{g^2}&=\frac{2(1-\sin\theta)(r-R_0)}{R_1-R_0}\frac{x_2^2}{\sqrt{(1+x_2^2)(\lambda^2+x_2^2)}}\\
&=\frac{2(1-\sin\theta)rx_2^2}{(R_1-R_0)\sqrt{(1+x_2^2)(\lambda^2+x_2^2)}}
-\frac{2(1-\sin\theta)R_0x_2^2}{(R_1-R_0)\sqrt{(1+x_2^2)(\lambda^2+x_2^2)}}\\
&=\textcircled{5}'-\textcircled{5}''
\end{align*}
and
\begin{align*}
\textcircled{5}'-\frac{\textcircled{7}}{g^2}&=\frac{2(1-\sin\theta)rx_2^2}{(R_1-R_0)\sqrt{(1+x_2^2)(\lambda^2+x_2^2)}}
-\frac{2(1-\sin\theta)x_2^2\sqrt{1+x_2^2}}{(R_1-R_0)|x|\sqrt{\lambda^2+x_2^2}}\\
&=\frac{2(1-\sin\theta)x_2^2}{r(R_1-R_0)\sqrt{(1+x_2^2)(\lambda^2+x_2^2)}}(r^2-1-r^2\sin^2\theta)\\
&\geq\frac{-2}{R_0(R_1-R_0)}.\\
\end{align*}
Moreover, it's easy to see that
\[\textcircled{5}''\leq \frac{2R_0}{R_1-R_0}\]
and
\[\frac{\textcircled{4}}{g^2}\leq\frac{1}{R_0^2}.\]
Combining these inequalities we get, when $R_0>\frac{10}{\lambda}$ and $R_1>\frac{10}{\lambda^2}R_0+R_0$
\begin{align*}
&\frac{1}{g^2}(\textcircled{1}-\textcircled{2}-\textcircled{3}-\textcircled{4}+\textcircled{5}-\textcircled{7})\\
&\geq\lambda^2-\frac{1}{R_0^2}-\frac{1}{(R_1-R_0)^2}-\frac{1}{R_0^2}-\frac{2}{R_0(R_1-R_0)}-\frac{2R_0}{R_1-R_0}>0.
\end{align*}
Therefore, we proved that in the region $R_0<|x|<R_1,$ $\psi$ is spacelike. This completes the proof of Lemma \ref{lc1lem3}.
\end{proof}

From the discussion in Subsection \ref{semitroughs} we know that, for every ball $\bar{\B}\subset \dS^{n-1},$
we can first apply Lorentz transform to $\psi,$ then rotate the frame to obtain a new spacelike function $\psi_{\bar{\B}},$
such that its image of the Gauss map is the convex hull of $\bar{\B}$ in $B_1.$ Moreover, by Lemma \ref{lc1lem2} and \ref{lc1lem3},
it's clear that $\psi_{\bar{\B}}\geq \z^1_{\bar{\B}}$ as $|x|\goto\infty.$
Recall that the upper barrier of our supersolution is
$\uu_1(x)=\inf\limits_{\F\subset\bar{\B}, \delta(\bar{\B})\leq\pi-\delta_0}z^1_{\bar{\B}}(x).$
We define
$\psi_1=\inf\limits_{\F\subset\bar{\B}, \delta(\bar{\B})\leq\pi-\delta_0}\psi_{\bar{\B}}(x),$
then when $|x|\goto\infty,$ we have $\psi_1\geq\uu_1.$
Furthermore, when $|x|\leq R_0,$ it's easy to see that $\psi_1(x)=\sqrt{\lambda^2+V^2_{\F}(x)}.$
Now, let $K\subset\R^n$ be a compact set, by Theorem 4.3 of \cite{BS} we know there exists $\delta>0$ such that $\lu(x)-V_\F(x)\geq 2\delta$ on $K.$
We will choose $R_0$ so large that $K\subset\{x\in\R^n|\,|x|<R_0\}.$ Then, we set $\lambda$ small
such that $\lu(x)-(\psi_1+\frac{\delta}{2})\geq \delta$ on $K.$ From the discussion above we know
$(\psi_1+\frac{\delta}{2})-\uu_1\geq\frac{\delta}{2}$ as $|x|\goto\infty.$ Smoothing $\psi_1$ by a standard convolution. Applying Lemma \ref{lc1lem1}
we get the local $C^1$ estimate on $K$ for every spacelike convex function $u$ between $\lu$ and $\uu_1(x).$

\subsection{Local $C^2$ estimates}
\label{lc2}
Without loss of generality, we may assume $\lu(x), h(x)\goto\infty$ as $|x|\goto\infty.$ For if they don't, since the image of the Gauss map of $\M_{\lu}, \M_{h}$ are the same, we can always apply Lorentz transform to
$\M_{\lu}$ and $\M_{h},$ such that the resulting barrier functions $\td{\lu}$ and $\td{h}$ satisfy $\td{\lu}(x), \td{h}(x)\goto\infty$ as
$|x|\goto\infty.$ This is equivalent to cut $\M_{\uj}$ with a tilted plane.
\begin{lemm}
\label{lc2lem1}
Let $\ujs$ be the solution of \eqref{cs1.2}, $\uj$ be the Legendre transform of $\ujs,$ and $\Omega_J=D\ujs(\td{F}_J).$
For any giving $s>1,$ let $J_s>0$ be a positive number such that when $J>J_s,$ $\uj|_{\T\Omega_J}>s.$ Let $\ka_{\max}(x)$
be the largest principal curvature of $\M_{\uj}$ at $x,$ where $\M_{\uj}=\{(x, \uj(x))|x\in\Omega_J\}.$
Then, for $J>J_s$ we have
\[\max\limits_{\M_{\uj}}(s-\uj)\ka_{\max}\leq C_5.\]
Here, $C_5$ only depends on the local $C^1$ estimates of $\uj$.
\end{lemm}
\begin{proof} For our convenience, we will omit the supscript $J.$ The basic idea of the proof comes from \cite{GRW}. Let's consider the test function
\begin{eqnarray}\label{11.1}
\varphi&=&m\log(s-u)+\log P_m-mN\langle \nu,\mathbf{E}\rangle,
\end{eqnarray}
where $P_m=\sum_j\kappa_j^m,$ $\mathbf{E}=(0,\cdots, 0, 1),$ and
$N, m>0$ are some undetermined constants. Suppose that the
function $\varphi$ achieves its maximum value on $\M$ at some
point $x_0$. We may choose a local orthonormal frame $\{\tau_1, \cdots, \tau_n\}$
such that at $x_0,$ $h_{ij}=\ka_i\delta_{ij}$ and $\ka_1\geq\ka_2\geq\cdots\geq\ka_n.$
Differentiating $\varphi$ twice at $x_0$, we have,
\begin{equation}\label{3.2}
\dfrac{\dsum_j\kappa_j^{m-1}h_{jji}}{P_m}-Nh_{ii}\langle X_i,\mathbf{E}\rangle+\frac{\langle X_i,\mathbf{E}\rangle}{s-u}=0,
\end{equation}
and,
\begin{align}\label{3.3}
0\geq &\dfrac{1}{P_m}[\dsum_j\kappa_j^{m-1}h_{jjii}+(m-1)\dsum_j\kappa_j^{m-2}h_{jji}^2
+\dsum_{p\neq q}\dfrac{\kappa_p^{m-1}-\kappa_q^{m-1}}{\kappa_p-\kappa_q}h_{pqi}^2] \\
&-\dfrac{m}{P_m^2}(\dsum_j\kappa_j^{m-1}h_{jji})^2 -Nh_{imi}\langle X_m,\mathbf{E}\rangle
-Nh_{ii}^2\langle \nu,\mathbf{E}\rangle\nonumber\\
&+\frac{h_{ii}\langle \nu,\mathbf{E}\rangle}{s-u}-\frac{u_i^2}{(s-u)^2}.\nonumber
\end{align}
\par
Note that $u$ satisfies equation \eqref{int1.0}. Now, let's differentiate equation \eqref{int1.0} twice and obtain
\[\sigma_k^{ii}h_{iij}=0\,\,\mbox{and $\sigma_k^{ii}h_{iijj}+\sigma_k^{pq, rs}h_{pqj}h_{rsj}=0.$}\]
Recall that in Minkowski space we have
\[h_{jjii}=h_{iijj}+h_{ii}^2h_{jj}-h_{ii}h^2_{jj}.\]
Therefore,
\be\label{c2g1.5}
\begin{aligned}
0&\geq\frac{1}{P_m}\lt[\sum_j\ka_j^{m-1}\sigma_k^{ii}(h_{iijj}+h_{ii}^2h_{jj}-h_{ii}h_{jj}^2)\right.\\
&\left.+(m-1)\sum_j\ka_j^{m-2}\s_{k}^{ii}h_{jji}^2+\sum\limits_{p\neq q}\frac{\ka_p^{m-1}-\ka_q^{m-1}}{\ka_p-\ka_q}\s_k^{ii}h_{pqi}^2\rt]\\
&-\frac{m}{P_m^2}\s_k^{ii}\lt(\sum_j\ka_j^{m-1}h_{jji}\rt)^2-N\s_k^{ii}\ka_i^2\lt<\nu, \mathbf{E}\rt>
+\frac{k\binom{n}{k}\lt<\nu, \mathbf{E}\rt>}{s-u}-\frac{\sigma_k^{ii}u_i^2}{(s-u)^2}\\
&\geq\frac{1}{P_m}\lt\{\sum_j\ka_j^{m-1}\lt[-k\s_kh_{jj}^2+K(\s_k)_j^2-\s_k^{pq, rs}h_{pqj}h_{rsj}\rt]\right.\\
&\left.+(m-1)\s_k^{ii}\sum_j\ka_j^{m-2}h_{jji}^2+\s_k^{ii}\sum\limits_{p\neq q}\frac{\ka_p^{m-1}-\ka_q^{m-1}}{\ka_p-\ka_q}h_{pqi}^2\rt\}\\
&-\frac{m}{P_m^2}\s_k^{ii}\lt(\sum_j\ka_j^{m-1}h_{jji}\rt)^2-N\s_k^{ii}\ka_i^2\lt<\nu, \mathbf{E}\rt>+\frac{k\binom{n}{k}\lt<\nu, \mathbf{E}\rt>}{s-u}-\frac{\sigma_k^{ii}u_i^2}{(s-u)^2}.
\end{aligned}
\ee
Here, $K$ is some sufficiently large constant. Note that
\[-\s_k^{pq, rs}h_{pqj}h_{rsj}=\sum_{p, q}\s_k^{pp, qq}h_{pqj}^2-\sum_{p, q}\s_k^{pp, qq}h_{ppj}h_{qqj},\]
we denote
\[A_i=\frac{\ka_i^{m-1}}{P_m}[K(\s_k)_i^2-\sum_{p, q}\s_k^{pp, qq}h_{ppi}h_{qqi}],\]
\[B_i=\frac{2\ka_j^{m-1}}{P_m}\sum_j\s_k^{jj, ii}h_{jji}^2,\]
\[C_i=\frac{m-1}{P_m}\s_k^{ii}\sum_j\ka_j^{m-2}h^2_{jji},\]
\[D_i=\frac{2\s_k^{jj}}{P_m}\sum\limits_{j\neq i}\frac{\ka_j^{m-1}-\ka_i^{m-1}}{\ka_j-\ka_i}h^2_{jji},\]
and
\[E_i=\frac{m}{P_m^2}\s_k^{ii}\lt(\sum_j\ka_j^{m-1}h_{jji}\rt)^2.\]
Then equation \eqref{c2g1.5} becomes
\be\label{c2g1.6}
0\geq \sum_i(A_i+B_i+C_i+D_i-E_i)-\frac{k\binom{n}{k}\sum_j\ka_j^{m+1}}{P_m}-N\s_k^{ii}\ka_i^2\lt<\nu, \mathbf{E}\rt>
+\frac{k\binom{n}{k}\lt<\nu, \mathbf{E}\rt>}{s-u}-\frac{\sigma_k^{ii}u_i^2}{(s-u)^2}.
\ee
By Lemma 8 and 9 in \cite{LRW} we can assume the following claim holds.
\begin{claim}
For any $i=1, 2, \cdots, n$ we have
\be\label{c2g1.7}
A_i+B_i+C_i+D_i-(1+\frac{\eta}{m})E_i\geq 0,
\ee
where $m>0$ is sufficiently large and $0<\eta<1$ is small.
\end{claim}
Here we note that, by Lemma 8 of \cite{LRW}, for $i=2, 3, \cdots, n,$ inequality \eqref{c2g1.7}
always holds. In particular, for $i=2, 3, \cdots, n$ we have
 \[A_i+B_i+C_i+D_i-(1+\frac{1}{m})E_i\geq 0.\]
 For $i=1,$ if \eqref{c2g1.7} doesn't hold, by Lemma 9 of \cite{LRW}, there would exist a $\delta>0$ small such that
$\ka_k\geq\delta\ka_1.$ Since
\[\sigma_k(\ka[\M_{\uj}])=\binom{n}{k}\geq \ka_1\times\cdots\times\ka_k\geq \delta^{k-1}\ka_1^k,\]
we would obtain an upper bound for
$\ka_1$ directly, then we would be done.

Combining equation \eqref{c2g1.7} with \eqref{c2g1.6} we get
\begin{eqnarray}\label{3.37}
0&\geq&-C\kappa_{1}
+\sum_{i=2}^n\frac{\sigma_k^{ii}}{P_m^2}(\sum_j
\kappa_j^{m-1}h_{jji})^2 -N \sigma_k^{ii}\kappa_i^2\langle\nu,\mathbf{E}\rangle
+\dfrac{ k\binom{n}{k}\langle \nu,\mathbf{E}\rangle}{s-u}-\dfrac{\sigma_k^{ii}u_i^2}{(s-u)^2}.
\end{eqnarray}
By \eqref{3.2}, we have, for  any fixed $i\geq 2$,
$$
-\frac{\sigma_k^{ii} u_i^2}{(s-u)^2} =-\frac{\sigma_k^{ii}}{
P_m^2}(\sum_j \kappa_j^{m-1}h_{jji})^2
+\sigma_{k}^{ii}N^2u_i^2h_{ii}^2-\frac{2N\sigma_{k}^{ii}u_i^2h_{ii}}{s-u}.
$$
Hence, \eqref{3.37} becomes,
\begin{eqnarray}\label{3.38}
0&\geq&-C\kappa_1+\sum_{i=2}^n\left(\sigma_{k}^{ii}N^2u_i^2h_{ii}^2-\frac{2N\sigma_{k}^{ii}u_i^2h_{ii}}{s-u}\right)\\&&
-N \kappa_{i}^2\sigma_k^{ii}\langle \nu,\mathbf{E}\rangle
+\frac{k\binom{n}{k}\langle\nu,\mathbf{E}\rangle}{s-u}-\frac{\sigma_k^{11}u_1^2}{(s-u)^2}.\nonumber
\end{eqnarray}
Since, there is some positive constant $c_0$ such that, $$h_{11}\sigma_k^{11}\geq c_0>0,$$ we have,
\begin{eqnarray}\label{3.39}
0&\geq&\left(-\frac{c_0N\langle \nu,\mathbf{E}\rangle}{2}-C\right)\kappa_1-\sum_{i=2}^n\frac{2N\sigma_{k}u_i^2}{s-u}
-\frac{N}{2}\sigma_k^{11}\kappa_1^2\langle \nu,\mathbf{E}\rangle
+\frac{k\binom{n}{k}\langle \nu,\mathbf{E}\rangle}{s-u}-\frac{\sigma_k^{11}u_1^2}{(s-u)^2}.\nonumber
\end{eqnarray}
Here, we have used for any $1\leq i\leq n$ (no summation),
$$\sigma_k=\kappa_i\sigma_k^{ii}+\sigma_k(\kappa|i)\geq
\kappa_i\sigma_k^{ii}.$$
On the other hand, it's easy to see that
$$g^{ij}u_iu_j=|Du|^2+\frac{|Du|^4}{1-|Du|^2}=\frac{|Du|^2}{1-|Du|^2},$$ which implies
$$|\nabla u|_g<-\langle\nu,\mathbf{E}\rangle=\frac{1}{\sqrt{1-|Du|^2}}.$$
Hence, we obtain, for $-N\langle \nu,\mathbf{E}\rangle\geq
\dfrac{4C}{c_0}$,
\begin{equation}\label{3.40}
\left(\frac{C}{s-u}+\frac{C\sigma_k^{11}}{(s-u)^2}\right)\lt(-\lt<\nu,\mathbf{E}\rt>\rt)^2\geq\frac{Nc_0}{4}\kappa_1\lt(-\lt<\nu,\mathbf{E}\rt>\rt)
+\frac{N}{2}\sigma_k^{11}\kappa_1^2\lt(-\lt<\nu, \mathbf{E}\rt>\rt).
\end{equation}
If at the maximum value point $x_0$, $s-u\geq \sigma_k^{11}$,  the above
inequality becomes,
$$(-\langle\nu,\mathbf{E}\rangle)\frac{2C}{s-u}\geq\frac{Nc_0}{4}\kappa_1,
$$
which implies that at the point $x_0$, we have
 $$(s-u)\kappa_1\leq C.$$
If $s-u\leq \sigma_k^{11}$,  the inequality becomes,
$$(-\langle\nu,\mathbf{E}\rangle)\frac{2C\sigma_k^{11}}{(s-u)^2}\geq\frac{N}{2}\sigma_k^{11}\kappa_1^2,
$$
 which also implies that at the point $x_0$, we have
 $$(s-u)^2\kappa_1^2\leq C.$$
 Therefore, we obtain the desired Pogorelov type $C^2$ local estimates.
\end{proof}
A direct consequence of Lemma \ref{lc2lem1} is the following nonexistence result.
\begin{coro}
\label{lc2cor1}
Suppose $\M_u=\{(x, u(x))| x\in\R^n\}$ is an entire, convex, spacelike hypersurface with constant $\s_k$ curvature, namely, it satisfies the equation
$$\sigma_k(\ka[\M_u])=\binom{n}{k}.$$ Moreover, we assume $\M_u$ is strictly spacelike, that is, there is some constant $\theta<1$ such that
$$|Du|\leq \theta<1,\,\,x\in\R^n.$$ Then, such $\M_u$ does not exist.
\end{coro}
\begin{proof}
Notice that, if such $\M_u$ does exist, we can always apply Lorentz transform to $\M_u$ such that the resulting function
$\td{u}$ satisfies $\td{u}(x)\goto\infty$ as $|x|\goto\infty.$ Without loss of generality, in the following, we will always assume
$u(x)\goto\infty$ as $|x|\goto\infty.$

Now, for any point $p\in T_s=\{x\in\R^n, u(x)<s\},$ by Lemma \ref{lc2lem1} we have
\[(s-u)\ka_{\max}(p)\leq C(\theta).\]
Therefore, for any $x\in\R^n,$ we may take $s>0$ large such that $x\in T_{s/2}.$ Then, we get
\[\ka_{\max}(x)\leq\frac{2C(\theta)}{s}.\]
Letting $s\goto\infty$ leads to a contradiction.
\end{proof}
\begin{rmk}
The above Corollary can be seen as a generalization of the rigidity theorem obtained by Aiyama \cite{Aiya}, Xin \cite{Xin}, and Palmer \cite{Palmer}.
\end{rmk}

\subsection{Convergence of $\uj$ to the strictly convex solution $u$}
\label{scv}
Recall that in Section \ref{cs}, we proved that there exists a sequence of  strictly convex solution $\{\ujs\}_{J\in\mathbb{N}}$ to the approximating equations
\be\label{scv}
\left\{
\begin{aligned}
F(\w\gas_{ik}\ujs_{kl}\gas_{lj})&=\frac{1}{\binom{n}{k}^{\frac{1}{k}}}\,\,\text{in $\td{F}_J$}\\
\ujs&=\vjs\,\,\text{on $\partial\td{F}_J,$}
\end{aligned}
\right.
\ee
where $\vjs=\lus|_{\partial\td{F}_J},$ and $\{\td{F}_J\}_{J\in\mathbb{N}}$ is a sequence of strictly convex set satisfying $\td{F}_J\subset\td{F}_{J+1}\subset\td{F}$ and $\p\td{F}_J$ is smooth. Let $\uj$ denote the Legendre transform of $\ujs.$ Then
$\uj$ satisfies
\[\sigma_k^{\frac{1}{k}}(\la[\M_{\uj}])=\binom{n}{k}^{\frac{1}{k}}.\]
Combining estimates in Subsections \ref{lc0} - \ref{lc2} with the classic regularity theorem, we know that
there exists a subsequence of $\{\uj\}_{J=1}^\infty,$
which we will still denote by $\{\uj\}_{J=1}^\infty,$ converging locally smoothly to a convex function $u$ defined over $\R^n,$ and $u$ satisfies
\[\sigma_{k}(\ka[\M_u])=\binom{n}{k}\,\,\mbox{and $\lu(x)<u(x)<h(x),$ for $x\in\R^n$}.\]
Since when $\frac{x}{|x|}\in\F,$ $\lu(x)$ and $h(x)\goto |x|$ as $|x|\goto\infty,$ it's easy to see that $u(x)$ satisfies $u(x)\goto|x|$ for $\frac{x}{|x|}\in\F$ as $|x|\goto\infty.$

In order to finish the proof of Theorem \ref{intth1.1}, we only need to prove $u$ is strictly convex. By a small modification of Theorem 1.2 in \cite{GLM} (see also \cite{CGM}), we obtain the following Minkwoski space version of Constant Rank Theorem:
\begin{theo}
\label{constant rank}
Suppose $\Gamma\subset \mathbb{R}^{n+1}\times \mathbb{H}^n $ is a bounded open set. Let $\psi\in C^{1,1}(\Gamma)$ and $\psi(X,y)^{-1/k}$ be locally convex in the $X$ variable for any fixed $y\in\mathbb{H}^n$. Let $\M$ be an oriented, immersed, connected, spacelike hypersurface in Minkowski space $\mathbb{R}^{n,1}$ with a nonnegative definite second fundamental form. If $(X,\nu(X))\in\Gamma$ for  each $X \in \M$ and the principal curvatures $\kappa=(\kappa_1,\kappa_2,\cdots,\kappa_n)$ of $\M$ satisfies the equation
\be\label{cr1}\s_k(\kappa[\M])=\psi(X,\nu),\,\,1\leq k\leq n,
\ee
then the second fundamental form of $\M$ is of constant rank.
\end{theo}
 Note that despite in the Minkowski space, the Gauss equation has an opposite sign, one can verify this theorem step by step following the argument of Theorem 1.2 in \cite{GLM} (see also Theorem 2 in \cite{CGM}). Therefore, we skip the proof here.

 \begin{theo}
 \label{splitting}
Let $\M$ be a convex, spacelike hypersurface satisfying \eqref{cr1}. If $\M$ is not strictly convex,
then after an $\R^{n,1}$
rigid motion, $\R^{n,1}$
splits as a product $\R^{l,1}\times\R^{n-l},$ $l\geq k,$  such that $\M$ also splits as a
product $\M^{l}\times\R^{n-l}.$ Here $\M^{l}\subset\R^{l,1}$
is a strictly convex, $l$-dimensional graph whose $\s_k$ curvature is equal to $\psi$.
 \end{theo}
\begin{proof}
Let $W$ be the Weigarten map of $\M$ and $\text{ker}(W)=\{v\in T\M|Wv=0\}$ be the kernel of $W,$ that is, the eigenvector space corresponding to the zero principal curvature. In view of Theorem \ref{constant rank} we know that, the dimension of $\text{ker}(W)$ is a constant. Without loss of generality, let's assume it to be $n-l$ for some $n>l\geq k.$

$\bf{Step 1.}$ We first prove the $\text{ker}(W)$ is a smooth subbundle of the tangent bundle $T\M$. The smoothness can be viewed as follows. Choosing any smooth orthonomal frame $\{e_1,e_2,\cdots, e_n\}$, the matrix of the map $W$ can be expressed by $(h_{ij})$ in this frame. We assume the first $l$ rows and $l$ columns of $W$ are linearly independent, and let
$$E_m=\sum_{i=1}^la_ie_i+e_m,$$ for $l+1\leq m\leq n.$ Using $WE_m=0$, we can find smooth linearly independent vector fields $\{E_{l+1},\cdots, E_n\}$
such that $\text{span}\{E_{l+1}, \cdots, E_n\}=\text{ker}(W)$. Thus, $\text{ker}(W)$ is a smooth subbundle of $T\M$.

$\bf{Step 2.}$ Next, we want to show the Frobenius condition is satisfied by $\text{ker}(W).$
Let $\{e_1,\cdots, e_n\}$ be some orthonomal frame such that
$$\text{span}\{e_{l+1}, \cdots, e_{n}\}=\text{ker}(W).$$ We still denote the matrix of $W$ by $(h_{ij}).$ Then, it's clear that $h_{im}=h_{mi}=0$ for any $1\leq i\leq n$ and $l+1\leq m\leq n$. By a proper rotation of the first $l$ vectors, we may assume at a fixed point $P\in\M$, $(h_{ij})$ is diagonal, i.e., $h_{ij}=\kappa_i\delta_{ij}$ and $\ka_i>0$ iff $1\leq i\leq l.$
Notice that $\sigma_{l+1}(W)=\sigma_{l+2}(W)=0$, the covariant derivative of this two functions with respect to $e_m,$ $1\leq m \leq n$ are
\begin{eqnarray}
0&=&(\sigma_{l+1})_m=\sigma_{l+1}^{ii}h_{iim}=\kappa_1\cdots\kappa_l\sum_{i=l+1}^nh_{iim},\nonumber\\
0&=&(\sigma_{l+2})_{mm}=\sigma_{l+2}^{ii}h_{iimm}+\sigma_{l+2}^{pq,rs}h_{pqm}h_{rsm}.\nonumber
\end{eqnarray}
Therefore, we obtain
\be\label{split1}
0=\sum_{i=l+1}^nh_{iim},
\ee
\be\label{split2}
0=\sum_{p\neq q, p,q>l }h_{ppm}h_{qqm}-\sum_{p\neq q, p,q>l}h_{pqm}^2.
\ee
Note that $\eqref{split1}^2-\eqref{split2}=0,$ we have
$$h_{iim}=0, \text{ for } i>l.$$ This in turn yields $$h_{pqm}=0,$$ for any $p,q>l$ and any $1\leq m\leq n$.

On the other hand, for $p,q>l$ and $1\leq m\leq n$, by $h_{mp}=0$, we get
$$e_q(h_{mp})=0.$$ Denote the connection of $M$ by $\nabla,$ then in view of the definition of the covariant derivatives we can see,
$$h_{mpq}=e_q(h_{mp})-h(\nabla_{e_q}e_m, e_p)-h(e_m, \nabla_{e_q}e_p).$$ Using $h_{mpq}=0$, we have
$$\sum_sh_{ms}\langle\nabla_{e_q}e_p, e_s\rangle=0,$$  which gives, for $1\leq m\leq l$,
\be\label{split3}
\langle\nabla_{e_q}e_p, e_m\rangle=0.
\ee
Therefore, $\nabla_{e_q}e_p\in\text{ker}(W)$ and the Frobenius condition is satisfied by $\text{ker}(W).$

$\bf{Step 3}.$ Finally, let $\M_0$ be the integral manifold of $\text{ker}(W),$ in this step, we will show $\M_0$ is flat.

By \eqref{split3} we know that $\M_0$ is a totally geodesic
$(n-l)$-dimensional submanifold of $\M.$ Moreover, it's not hard to see that $\M_0$ lies in the hyperplane $\mathbb{P}$ that is perpendicular to $\nu,$
where $\nu$ is the timelike unit normal of $\M.$ We can choose a coordinate such that
$\mathbb{P}=\{x|x_{n+1}=\lt<x, E\rt>=0\}$ for $E=(0,\cdots, 0, 1),$ then we have $\M_0\subset\mathbb{P}.$
Next, let's denote $\td{\M}_0:=\M_0\times\R^{l-1}$ be an $(n-1)$ dimensional hypersurface in $\mathbb{P}.$
Then at each point of $\td{\M}_0$ we have, 
$$\text{span}\{e_{l+1}, \cdots, e_{n}, \mu_1, \cdots, \mu_{l-1}\}=T\td{\M}_0$$
for some fixed orthonormal vectors $\mu_1, \cdots, \mu_{l-1}.$ Denote the normal of $\td{M}_0$ in $\mathbb{P}$ by $\td{\nu}.$
Recall \eqref{split3} we can see that $D_{e_\alpha}e_\beta\in\text{span}\{e_{l+1}, \cdots, e_{n}, \nu\},$ $l+1\leq\alpha, \beta\leq n.$
Therefore, for any $l+1\leq\alpha, \beta\leq n$ and $1\leq j\leq l-1,$ we have
\[D_{e_\alpha}e_\beta\cdot\td{\nu}=0\,\,\mbox{and $D_{e_\alpha}\mu_j\cdot\td{\nu}=0$}.\]
We conclude that $\td{\M}_0$ is flat, which implies $\M_0$ is flat. By Cheeger-Gromoll splitting theorem (see Theorem 2 in \cite{CG}),
we complete the proof of Theorem \ref{splitting}.
 \end{proof}

 By Theorem \ref{constant rank}, we can see that if our solution $u$ of \eqref{int1.0} has some degenerate point $x_0\in R^n$, i.e., $\sigma_{l+1}(\ka[\M(x_0)])=0$,
 then $\s_{l+1}(\ka[\M])\equiv 0$. Applying Theorem \ref{splitting}, we conclude that $\M_u=\{(x,u(x))| x\in\R^n\}$ splits
 into $\M^l\times\R^{n-l}$ , where $\M^l$ is an $l$-dimensional strictly convex hypersurface. This contradicts to the fact that
 when $\frac{x}{|x|}\in\F,$ as $|x|\goto\infty,$ $u\goto |x|.$ Therefore, we proved Theorem \ref{intth1.1}.

\subsection {Special case: Suppose $\F=\dS^{n-1}$}
\label{sp}
In \cite{RWX}, the following theorem is proved.
\begin{theo}
Given a $C^2$ function $\varphi$ on $B_1$, there is a unique strictly convex solution $\us\in C^{\infty}(B_1)\cap C^0(\bar{B}_1)$ to the equation
\be\label{int1.0*}
\left\{
\begin{aligned}
F(\w\gas_{ik}\us_{kl}\gas_{lj})&=1,\,\,\mbox{in $B_1$}\\
\us&=\vp,\,\,\mbox{on $\partial B_1$.}
\end{aligned}
\right.
\ee
Here $$\w=\sqrt{1-|\xi|^2},\ \ \gas_{ik}=\delta_{ik}-\frac{\xi_i\xi_k}{1+\w},\ \ \us_{kl}=\frac{\T^2\us}{\T\xi_k\T\xi_l},$$ $$F(\w\gas_{ik}\us_{kl}\gas_{lj})=\lt(\frac{\sigma_n}{\sigma_1}(\ka^*[\w\gas_{ik}\us_{kl}\gas_{lj}])\rt)^{\frac{1}{n-1}},$$
and $\ka^*[\w\gas_{ik}\us_{kl}\gas_{lj}]=(\ka^*_1, \cdots, \ka^*_n)$ are the eigenvalues of the matrix $(\w\gas_{ik}\us_{kl}\gas_{lj})$.
Moreover, the Legendre transform of $\us,$ which we will denote by $u$ satisfies
\[\sigma_{n-1}(\ka[\M_u])=1\,\,\mbox{and $\ka[\M_u]\leq C$}.\]
Here, $\M_u=\{(x, u(x)) |\, x\in\R^n\}$ is the spacelike graph of $u,$ $\ka[\M_u]$ denotes the principal curvatures of $\M_u,$ and the constant $C$ only depends on $|\varphi|_{C^2}$.
\end{theo}

Applying Subsection \ref{c2g} and Lemma \ref{lc2lem1}, we can easily generalize this theorem to the case when
$F(\w\gas_{ik}\us_{kl}\gas_{lj})=\lt(\frac{\sigma_n}{\sigma_{n-k}}(\ka^*[\w\gas_{ik}\us_{kl}\gas_{lj}])\rt)^{\frac{1}{k}}.$
Moreover, we want to point out that from Lemma \ref{lem-legendre-boundary} we can see
the Legendre transform of $\us$ which we denote by $u$ satisfies
$u(x)-|x|=-\vp\lt(\frac{x}{|x|}\rt).$ Therefore, Theorem \ref{intth1.2} is proved.

\end{document}